\documentclass[preprint]{imsart}
\usepackage[latin1]{inputenc}
\usepackage{amsmath,amssymb,amsthm} 
\usepackage{mathabx}
\usepackage{stmaryrd}
\SetSymbolFont{stmry}{bold}{U}{stmry}{m}{n}
\usepackage{bbm}
\usepackage{tikz}
\usepackage{graphicx}
\usepackage{nccmath}
\usepackage{enumerate}

\usepackage[pdfborder={0 0 0}]{hyperref}
\hypersetup{
  urlcolor = black,
  pdfauthor = {Walter Schachermayer, Florian Stebegg},
  pdfkeywords = {...},
  pdftitle = {...},
  pdfsubject = {...},
  pdfpagemode = UseNone
}

\def\presuper#1#2%
  {\mathop{}%
   \mathopen{\vphantom{#2}}^{#1}%
   \kern-\scriptspace%
   #2}

\newcommand\E{\mathbb{E}}
\newcommand\R{\mathbb{R}}
\newcommand\N{\mathbb{N}}
\renewcommand\P{\mathbb{P}}
\newcommand\G{\mathcal{G}}
\newcommand\T{\mathcal{T}}
\newcommand\F{\mathcal{F}}
\renewcommand\L{\mathcal{L}}

\newtheorem{theorem}{Theorem}[section]

\newtheorem{lemma}[theorem]{Lemma}
\newtheorem{proposition}[theorem]{Proposition}
\newtheorem{numerics}[theorem]{Numerical Evidence}
\theoremstyle{definition}

\newtheorem{remark}[theorem]{Remark}

\begin{document}

\begin{frontmatter}
\title{The Sharp Constant for the Burkholder-Davis-Gundy Inequality and Non-Smooth Pasting}
\runtitle{The Sharp Constant for the BDG inequality}
\runauthor{W. Schachermayer and F.Stebegg}

\begin{aug}
\author{\fnms{Walter} \snm{Schachermayer} \thanksref{a1,e1,t1}\ead[label=e1,mark]{walter.schachermayer@univie.ac.at}}
\and
\author{\fnms{Florian} \snm{Stebegg} \thanksref{a2,e2,t2}\ead[label=e2,mark]{florian.stebegg@columbia.edu}}
\address[a1]{Fakult\"at f\"ur Mathematik, Universit\"at Wien \\ Oskar-Morgenstern-Platz 1, Room 06.131\\ A-1090 Wien.\\ \printead{e1}}
\address[a2]{Department of Statistics, Columbia University \\ 1255 Amsterdam Avenue, Room 906\\ 10027 New York, NY.\\ \printead{e2}}

\thankstext{t1}{Supported in part by the Austrian Science Fund (FWF) under grant P25815 and P28661 and by the Vienna Science and Technology Fund (WWTF) under grant MA14-008.}
\thankstext{t2}{Supported in part by the Austrian Science Fund (FWF) under grants P26736 and Y782-N25.}

\end{aug}

\begin{abstract}
We revisit the celebrated family of BDG-inequalities introduced by Burkholder, Gundy
\cite{BuGu70} and Davis \cite{Da70} for continuous martingales.
For the inequalities $\E[\tau^{\frac{p}{2}}] \leq C_p \E[(B^*(\tau))^p]$ with $0 < p < 2$ we propose
a connection of the optimal constant $C_p$ with an
ordinary integro-differential equation which gives rise to a numerical method of finding
this constant. Based on numerical evidence we are able to calculate, for $p=1$, the explicit
value of the optimal constant $C_1$, namely $C_1 = 1,27267\dots$. In the course of our analysis, we find
a remarkable appearance of ''non-smooth pasting`` for a solution of a related ordinary integro-differential
equation.
\end{abstract}


\end{frontmatter}

\maketitle

\section{Introduction}

We consider the following version of the Burkholder-Davis-Gundy inequality \cite{BuGu70}, \cite{Da70}:

\begin{theorem}
\label{thm:main}
There is a constant $C>0$ such that, for every bounded stopping time $\tau$, we have
\begin{align}\label{V1}
\mathbb{E}\Big[\tau^{\frac{1}{2}}\Big] \leq C \, \mathbb{E} [B^*(\tau)].
\end{align}
\end{theorem}

Here $(B(t))_{t \geq 0}$ denotes a standard Brownian motion, starting at $B(0)=0.$ By $B^*(t)$ we denote the corresponding running maximum of the absolute value
\begin{align*}
B^*(t):=\sup_{0 \leq u \leq t} |B(u)|.
\end{align*}
It is obvious that the set of constants $C$ which satisfy inequality \eqref{V1} is a closed, unbounded interval in $\R_+$. By the results of \cite{BeSi15} it is known that $C=\frac{3}{2}$ is contained in this set. To the best of our knowledge, this is the smallest constant known in the previous literature. In the present paper we establish the optimal value for this constant.

\begin{theorem} \label{thm:mainoide}
There is an ordinary integro-differential equation (see \eqref{V14a} below) depending on real parameters $C>0$ and $t_0>0$ such that $C$ satisfies \eqref{V1} if and only if there is $t_0$ such that this equation has a well-defined solution.

Numerical solutions of the equation \eqref{V14a} reveal that the smallest such $C$, i.e.~the optimal constant in the Burkholder-Davis-Gundy inequality \eqref{V1}, equals
$$\widehat{C} \approx 1,27267\dots.$$
\end{theorem}

The paper is organised as follows. As usual in stochastic control theory, we first introduce the \emph{value function} of the \emph{optimal stopping} problem which corresponds to the inequality \eqref{V1}. After some structural facts about the stopping problem we turn to some analytic properties of the value function in Section \ref{sec:valuecont}. We deduce the OIDE (ordinary integro-differential equation) which is 
referred to in Theorem \ref{thm:mainoide}. The subsequent section is devoted to properties of solutions to the
fundamental OIDE \eqref{V14a} which are needed to identify these solutions with the value function of the stopping problem in Section \ref{sec:synth}.

The critical $\widehat{t}_0 > 0$ associated to the optimal constant $\widehat{C}$ via \eqref{V14a} below also turns out to
be of somewhat independent interest: if $\rho$ denotes the first moment, say after $t=1$, when $t$ is bigger than
$cB^*(t)^2$, then $\E[\rho^{\frac{1}{2}}]$ is finite or infinite depending on whether $c$ is smaller or bigger than $\widehat{t}_0$
(Proposition \ref{prop:infexp} and \ref{prop:finexp}). In Section 6 we state a pointwise version of the BDG inequalities and in Section 7 we briefly
discuss the case of general $0<p<2$ without entering into a numerical analysis. Finally, in Section 8 we discuss the fact
why the constant $\widehat{C} = \sqrt{3}$ which was established by D. Burkholder \cite{Bu02} as the optimal constant for
\eqref{V1} in the case of martingales which are not necessarily continuous, is different from the present constant
$\widehat{C}=1,27267\dots$ which holds true for continuous processes. We relate this discrepancy with a certain
lack of concavity of the value function.

\section{The Value Function of an Optimal Stopping Problem} \label{sec:value}

Fix a constant $C>0.$ Following a well-known path in optimal control theory we define the value function 
\begin{align}\label{V4}
V(t,b,b^*):=\sup_{\tau \in \mathcal{T}(t)} \mathbb{E}^{(t,b,b^*)} [\tau^{\frac{1}{2}}-CB^*(\tau)],
\end{align} 
where $\mathcal{T}(t)$ denotes the set of bounded stopping times $\tau \geq t$ and $\mathbb{E}^{(t, b, b^*)}$ denotes the expectation conditionally on starting the Brownian motion $B$ at time $t$ with the values $B_t=b, B^*_t=b^*.$ The domain of definition of $V$ is
\begin{align}\label{V4a}
D=\{(t,b,b^*): \quad 0 \leq t < \infty, \quad 0 \leq |b| \leq b^* < \infty \}.
\end{align}

Equivalently we can write
\begin{align}\label{eq:Vchar}
V(t,b,b^*):=\sup_{\tau \in \mathcal{T}} \mathbb{E}[\sqrt{t+\tau}-C(b^* \vee (b+B(\tau))^*],
\end{align}
which follows from the strong Markov property and stationarity of increments of Brownian motion.

Denote by $\widehat{C}$ the infimum of $C>0$ such that \eqref{V1} holds true. Clearly $\widehat{C}$ still satisfies \eqref{V1}. 
If $C < \widehat{C}$ then $V(t, b, b^*) \equiv \infty$, otherwise we have:

\begin{lemma}\label{lem:valueprops}
Let $C \geq \widehat{C}$ then $V$ defined via \eqref{V4} is
\begin{enumerate}[(i)]
\item continuous,
\item finite-valued, and
\item $t \mapsto V(t,b,b^*) - \sqrt{t}$ is decreasing for fixed $|b| \leq b^*$.
\end{enumerate}
In particular (ii) follows from the bounds
\begin{align}\label{p2}
\sqrt{t} + Cb^* \geq V(t,b,b^*) \geq \sqrt{t} - Cb^*.
\end{align}
\end{lemma}

\begin{proof}
The lower bound of $V$ follows from choosing the stopping time $t \in \mathcal{T}(t)$. For the upper bound observe that
we can estimate for an arbitrary $\tau \in \mathcal{T}$
\begin{align*}
\sqrt{t+\tau} - C(b^* \vee (b+B(\tau))^* &\leq \sqrt{t} + \sqrt{\tau} - C(b+B(\tau))^* \\
&\leq \sqrt{t} + \sqrt{\tau} - CB(\tau)^* + C|b| \\
&\leq \sqrt{t} + Cb^* + \sqrt{\tau} - CB(\tau)^*.
\end{align*}
Taking expectations we get the upper bound for $V$ from the representation in \eqref{eq:Vchar}.

Next observe that for $\tau \in \T$ we have for $t<t'$
\[\sqrt{t+\tau} - C(b^* \vee (b+B(\tau))^* - \sqrt{t} \leq \sqrt{t'+\tau} - C(b^* \vee (b+B(\tau))^* - \sqrt{t'}\]
by concavity of the square root. Now (iii) follows by taking expectations suprema.

For (i), please refer to Sections 7 and 9.2 in \cite{PeSh00}.
\end{proof}

To exclude the trivial case, we assume in the sequel that $C \geq \widehat{C}.$ 

For fixed $C$, the \emph{stopping region} $S \subseteq D$ and the \emph{non-stopping region} $NS \subseteq D$ are defined by
\begin{align}\label{p2.1}
S:=\{V=t^{\frac{1}{2}}- Cb^*\}, \quad NS:=\{V >t^{\frac{1}{2}}-Cb^*\}.
\end{align}

To characterize the stopping region $S$ first note that it is certainly not a good idea to stop when $|B(t)| < B^*(t)$.

\begin{lemma}\label{lem:hollow}
Let $(t,b,b^*) \in D$ with $|b| < b^*.$ Then $(t,b,b^*) \in NS.$
\end{lemma}

\begin{proof}
Consider the first exit time of the interval $[-b^*,b^*]$.
\end{proof}

Next we observe a useful scaling property of $V$ (compare Burkholder \cite{Bu02}).

\begin{lemma}\label{lem:scale}
For $a > 0$ and $(t,b,b^*) \in D$, we have
\begin{align}\label{p3}
V(a^2t,ab,ab^*) = aV(t,b,b^*).
\end{align}
\end{lemma}

\begin{proof}
This follows directly from the scaling property of Brownian motion: if $(B_t)_{t\geq 0}$ is a standard
Brownian motion, then $(a^{-1}B_{a^2t})_{t \geq 0}$ again is a standard Brownian motion. Also, a random time
$a^2\tau$ is a stopping time for the first process if and only if $\tau$ is a stopping time for the
second process.
\end{proof}

This allows us to derive the following Lemma where the first part is a direct consequence of Lemmas \ref{lem:valueprops} (iii) and
\ref{lem:scale} and the second part is technical and deferred to the appendix in Lemma \ref{lem:t0int}.
\begin{lemma}\label{lem:stopreg}
Let $0 \leq t \leq t'$ and $b \in \mathbb{R}$. Then $(t, b, |b|) \in S$ implies $(t', b, |b|) \in S$.

Hence, for fixed $C \geq \widehat{C}$ there is a smallest $t_0 \in [0, \infty]$ such that $(t,b,|b|) \in S$ if and only if 
\begin{align}\label{V6}
\frac{t}{|b|^2} \geq t_0.
\end{align}
In fact, we have $t_0 \in \,(0,\infty).$
\end{lemma}

The next result is a standard result in optimal control theory and also intuitively rather obvious. Again, the proof is
deferred to the appendix.

\begin{lemma}\label{lem:valuemartingale}
Suppose $C \geq \widehat{C}$ and let $(t, b, b^*)$ be in the non-stop region $NS$. Consider a Brownian motion $(B(u))_{t \leq u}$ starting at time $t$ conditionally on $B(t)=b$ and $B^*(t)=b^*.$ Let $\tau$ be the first hitting time of the stopping region $S$, i.e.
\begin{align}\label{L1}
\tau = \inf \{u \geq t: (u, B(u), B^*(u)) \in S\}
\end{align}
Then the value process stopped at time $\tau$
\begin{align} \label{form:stopproc}
M(u)^\tau:=V(u \wedge \tau, B(u \wedge \tau), B^*(u \wedge \tau)), \qquad u \geq t,
\end{align}
is a martingale.

The unstopped value process 
\begin{equation} \label{form:unstopproc}
M(u) := V(u,B(u),B^*(u)), \qquad u \geq t,
\end{equation} 
still is a supermartingale.
\end{lemma}

We conclude this section with a minor technical remark. In the above statement, as well as in most of the paper, we follow
the usual language of optimal control theory to condition on the event $\{B(t)=b, B^*(t) = b^*\}$. As this is a null set under
$\P$ this procedure needs some proper interpretation in order to make it rigorous. 

Let us now introduce some notation to make this a bit clearer.

We denote by $(\F(u))_{u \geq 0}$ the (right-continuous, saturated) filtration generated by the Brownian motion $(B(u))_{u \geq 0}$.
Of course, in definition \eqref{V4} the stopping time $\tau \in \T(t)$ is understood with respect to this filtration. But it is
clear from the Markov property that, for fixed $(t,b,b^*)$, we may assume that $\tau \in \T(t)$ depends only on the behavior of the Brownian motion $(B(u))_{u \geq t}$ after time $t$ and not on the previous behavior of $(B(u))_{0 \leq u \leq t}$
(except for the requirements $B(t)=b$ and $B^*(t)=b^*$).

To formalize this fact, we denote by $(\G^{(t)}(u))_{u \geq t}$ the (right-continuous, saturated) filtration generated
by $(B(u)-B(t))_{u \geq t}$. A stopping time $\tau \in \T(t)$ (i.e., with respect to the filtration $(\F(u))_{u \geq 0}$) then may also be considered as a randomized stopping time with respect to the filtration $(\G^{(t)}(u))_{u \geq t}$, the randomization given
by the trajectories of $(B(u))_{0 \leq u \leq t}$. As $(B(u))_{0 \leq u \leq t}$ is independent of the filtration
$(\G^{(t)}(u))_{u \geq t}$, we conclude that the value of \eqref{V4} does not change whether we optimize over the randomized
or the non-randomized stopping times with respect to the filtration $(\G^{(t)}(u))_{u \geq t}$. For an introduction
to the notion of randomized stopping times, please refer to \cite{BaCh77}

The bottom line of these considerations is that we may assume w.l.o.g. in \eqref{V4} that $\tau \in \T(t)$ is a stopping time with respect to the filtration $(\G^{(t)}(u))_{u \geq t}$.

Now, the statement Lemma \ref{lem:valuemartingale}
could be rephrased without referring to conditioning on a null set, by noting that $\tau$ is a stopping time with respect to the
filtration $(\G^{(t)}(u))_{t \leq u}$.

All other statements in the paper referring to conditioning on the values $B(t)$ and $B^*(t)$ could be made rigorous
in an analogous way if the reader insists, but we do not further elaborate on these technicalities.

\section{The Value Function from an Analytic Perspective} \label{sec:valuecont}
Again fix $C \geq \widehat{C}$. Differentiating the scaling equation \eqref{p3}
with respect to $a$ and setting $a=1$ we obtain, at least formally, the PDE
\begin{align}\label{V7a}
2tV_t + bV_b + b^*V_{b^*}=V.
\end{align}

The optimal constant $C$ for inequality \eqref{V1} will be determined by analyzing whether this PDE has 
a reasonable solution for given $C>0$ or not.

We need some preparation. For $0 < h < 1$ we denote by $f^h(s)$ the density of the distribution of the stopping time
$\rho^h=\inf \{t:|B(t)|=1\},$
where $B$ is a Brownian motion starting at $B(0)=1-h.$

Define
$$g(s)=\lim_{h \searrow 0} \frac{f^h(s)}{h}, \qquad s >0.$$

It is well-known (e.g. \cite[Exercise 2.2.8.11]{KaSh88}) that there is an explicit representation of $f^h(s)$ as an infinite sum. By differentiation of each summand we obtain an explicit infinite sum representation also for $g(s)$ (see the appendix below).

The function $g$ appears in the formulation of the subsequent lemma which will turn out to be of crucial relevance for our analysis.

\begin{lemma}\label{lem:valuederivatives}
Let $W: D \to \R^+$ be a continuous function such that 
\begin{enumerate}[(a)]
\item $b^* \mapsto W(t,b,b^*)$ is Lipschitz continuous and
\item $W(t,1,1) - \sqrt{t}$ is decreasing.
\end{enumerate}
Furthermore let $S \subseteq D$ be defined by
$S_W := \{(t,b,|b|) \in D: t/b^2 \geq t_W\}$ for some fixed $t_W$. Consider a standard Brownian Motion
$B(t)$ and define $\tau_W$ to be the first hitting time of $S_W$. 
Suppose that $X(t):= W(t,B(t),B^*(t))$ is a supermartingale and $X(t \wedge \tau_W)$ is a martingale. 
Further assume that the process $X(t \wedge \tau_W \wedge \sigma^h)$ is uniformly integrable
where $\sigma^h$ is given by $\sigma^h := \inf\{s \geq t: |B(s)|=1+h\}$.

Then,
\begin{enumerate}[(i)]
\item \begin{align}\label{V9}
W_{b^*} (t,1,1):= \lim_{h \searrow 0} \frac{1}{h} & [W(t,1,1+h) - W(t,1,1)]=0 \\
& \text{ for } 0 < t < t_W \nonumber
\end{align}
\item \begin{align}\label{p4}
W_{b}(t,1,1)&:= \lim_{h \searrow 0} \frac{1}{h} [W(t,1,1) - W(t,1-h,1]\\
&\phantom{:}= -\int^\infty_0 [W(t+s, 1, 1) - W(t,1,1)]g(s)ds \nonumber
\end{align}
\end{enumerate}
\end{lemma}

Observe that (i) in the above Lemma would follow directly from Ito's formula if we assume that $W$ is
sufficiently differentiable by considering
\begin{align}\label{p5}
dX(t) =(W_t + \tfrac{1}{2} W_{bb}) dt + W_b dB(t)+ W_{b^*}dB^*(t)
\end{align}
which is the increment of a martingale. The process $dB^*(t)$ is non-decreasing and its variation is a.s.~singular with respect 
to Lebesgue measure. A necessary condition for $(W(t,B(t), B^*(t)))_{\tau_W \geq t}$ to be a martingale therefore is that 
$W_{b^*}$ vanishes a.s.~with respect to the variation  measure of $dB^*$. This indicates that $W_{b^*}(t,b,b^*)=0$ should 
hold true whenever $|b|=b^*$ and $(t,b,b^*)$ is in the non-stop region $NS$. In particular, we should have $W_{b^*} (t,1,1)=0$, 
for $t<t_W.$

\begin{proof}[Proof of \ref{lem:valuederivatives}]
(i) For $h >0$ as in \eqref{V9} define, conditionally on $B(t)=1$ and $B^*(t)=1,$ the stopping times 
\begin{align}
\sigma^h&:=\inf\{ u \geq t:|B(u)| =1+h\} \qquad \text{and}\label{4.11}\\
\tau^h&:=\sigma^h \wedge \tau_W.\label{4.12}
\end{align}
Recall that the random variable $\tau_W$ is a stopping time with respect to
the filtration $(\G^{(t)}(u))_{t \leq u}$.
Note that the process $(u,B(u),B^*(u))_{u \geq t}$, starting at $B(t) = 1$ and $B^*(t)=1+h$, also remains in 
$D \backslash S_W$ up to the stopping time $\tau_W$. To see this, we distinguish two cases:
\begin{enumerate}[(i)]
\item $t \leq u < \tau^h$: Here we have $B(u) < 1+h=B^*(u)$ and thus $(u,B(u),B^*(u)) \in D \backslash S_W$.
\item $\tau^h \leq u < \tau_W$: In this case $B^*(u)$ is the same for $B^*(t)=1$ and $B^*(t)=1+h$.
\end{enumerate}

As $t < t_0$ we have that $\tau$ is a.s.~strictly positive. This implies that 
\begin{align}\label{A3b}
\lim_{h \searrow 0} \mathbb{P} [\sigma^h = \tau^h] =1,
\end{align}
as $\lim_{h \searrow 0} \sigma^h = t$, almost surely.

We may write 
\begin{align}\label{A4}
W(&t,1,1+h) - W(t,1,1) \\
&=\E^{(t,1,1+h)} [W(\tau^h , B (\tau^h ), B^*(\tau^h) )] - \E^{(t,1,1)} [W(\tau^h, B(\tau^h), B^*(\tau^h))] \nonumber \\
&=\E^{(t,1,1+h)} [W(\tau^h , B (\tau^h ), 1+h )] - \E^{(t,1,1)} [W(\tau^h, B(\tau^h), B^*(\tau^h))]. \nonumber
\end{align}

Here we use that $B^*(\tau^h) = 1+h$ conditionally on $B^*(t) = 1+h$ as $\tau^h \leq \sigma^h$ by definition. Furthermore
we make use of the martingale property of $W(t, B(t), B^*(t))$ and the optional stopping theorem.
For the use of the latter we need the assumption 
that $(W(u\wedge \tau^h, B(u \wedge \tau^h), B^*(u \wedge \tau^h)))_{u \geq t}$ is uniformly integrable.

On the set $\{\sigma^h = \tau^h\}$ we have $B^*(\tau^h)=B^*(\sigma^h)=1+h$ for both initial conditions $B(t)=1, B^*(t)=1+h$ and $B(t)=1, B^*(t)=1$. Therefore, the value $W(\tau^h, B(\tau^h), B^*(\tau^h))$ is the same under both initial conditions. It follows that
\begin{align*}
\mathbb{E}^{(t,1,1+h)}[&W(\tau^h, B(\tau^h), 1+h) \mathbbm{1}_{\{\sigma^h=\tau^h\}}] \\
&- \mathbb{E}^{(t,1,1)}[W(\tau^h, B(\tau^h), B^*(\tau^h)) \mathbbm{1}_{\{\sigma^h=\tau^h\}}]=0.
\end{align*}

On the remaining set $\{\tau^h < \sigma^h\}$ we have $B^*(\tau^h)\in [1,1+h]$. 
Because $W$ is Lipschitz continuous in the variable $b^*$ with some constant $L$ we may estimate
\begin{align*}
\mathbb{E}^{(t,1,1+h)}[&W(\tau^h, B(\tau^h), 1+h) \mathbbm{1}_{\{\tau^h<\sigma^h\}}] \\
&- \mathbb{E}^{(t,1,1)}[W(\tau^h, B(\tau^h), B^*(\tau^h) \mathbbm{1}_{\{\tau^h<\sigma^h\}}]
\leq Lh\mathbb{P}[\tau^h < \sigma^h].
\end{align*}

Dividing \eqref{A4} by $h$ and passing to the limit we obtain from \eqref{A3b} that 
$$W_{b^*}(t,1,1):= \lim_{h \searrow 0} \frac{1}{h} [W(t,1,1+h)-W(t,1,1)]=0$$

(ii) %
As $X$ is a martingale before hitting $S_W$ we have for $\rho^h$ as above that
\[W(t,1-h,1) = \E^{(t,1-h,1)}[W(\rho^h,1,1)] = \int_0^{\infty}W(t+s,1,1)f^h(s) ds,\]
where the density $f^h$ is given by 
\begin{align}\label{form:dens}
f^h(s) := \frac{1}{\sqrt{2\pi s^3}} \sum_{n=-\infty}^{\infty} \left[(4n+h) e^{-\frac{(4n+h)^2}{2s}} +(4n+2-h) e^{-\frac{(4n+2-h)^2}{2s}}\right]. 
\end{align}
We can use this relation to calculate the derivative w.r.t.\ the second component:

\[ -W_b(t,1,1) := \lim_{h \searrow 0} \frac{1}{h} \int_0^{\infty}[W(t+s,1,1)-W(t,1,1)]f^h(s)ds \]

We split this integral into two parts at some point $\alpha > 0$ and observe that $f^h(s)$ is continuous and
$f^h(s) \searrow 0$ for $h \searrow 0$ pointwise at $s >0$ and thus by Dini's Theorem also uniformly (monotone) on any 
interval $[\alpha,K]$, for $0 < h < h_0(\alpha)$. Therefore we have\
\begin{align} \label{form: derivative}
\lim_{h \searrow 0} \frac{1}{h} &\int_{\alpha}^{K}[W(t+s,1,1)-W(t,1,1)]f^h(s)ds \\
&= \int_{\alpha}^{K}[W(t+s,1,1)-W(t,1,1)]g(s)ds \nonumber
\end{align}
for $g$ given by 
\begin{align}\label{form:derdens}
g(s) := \lim_{h \searrow 0} \mfrac{f^h(s)}{h} = \frac{1}{\sqrt{2\pi s^3}} \left[1+
2\cdot\sum_{n=1}^{\infty}(-1)^n\left(1-\mfrac{(2n)^2}{s}\right)e^{-\tfrac{(2n)^2}{2s}}\right].
\end{align}
As before, because $u \mapsto W(u,1,1) - \sqrt{u}$ is decreasing and $s \mapsto \sqrt{t+s} - \sqrt{t}$ is
concave, we have $W(t+s,1,1)-W(t,1,1) \leq \sqrt{t+s}-\sqrt{t} \leq s \frac{1}{2\sqrt{t}}$.
Therefore the integrand is dominated by
$\frac{sf^h(s)}{2\sqrt{t}h}$.

For $K>0$ we can then estimate
\begin{align*}
\frac{1}{h}\int_K^{\infty}sf^h(s)ds &= \frac{1}{h}\P[\rho^h > K]\E[(B_{\rho^h}-B_0)^2|\rho^h > K] \\
&= \P[\rho^h > K](2-h) \leq 2\P[\rho^h > K].
\end{align*}
This probability tends to $0$ uniformly in $K$, thus the integrals over $[K,\infty)$ can be neglected and we can replace
$K$ by $\infty$ in \eqref{form: derivative}.

For the other part of the integral we first observe that
\begin{align*}
\lim_{h \searrow 0} \frac{1}{h}&\int_0^{\alpha} sf^h(s)ds \\
&=\lim_{h \searrow 0} \frac{1}{h} \int_0^{\infty} sf^h(s) ds - \frac{1}{h}\int_{\alpha}^{\infty} s f^h(s) ds \\
&=\lim_{h \searrow 0} \frac{1}{h} \E[\rho^h] - \frac{1}{h}\int_{\alpha}^{\infty} s f^h(s) ds \\
&=\lim_{h \searrow 0} 2 - h - \frac{1}{h}\int_{\alpha}^{\infty} s f^h(s) ds=2 - \int_{\alpha}^{\infty} s\,g(s) ds.
\end{align*}
The last integral converges to $2$ for $\alpha \to 0$ by monotone convergence.

We conclude by setting $M \geq |W_t(u,1,1)|$ for $u \in [t,t+\alpha]$, making the estimate

\begin{align*}
\left| \frac{1}{h}\int_0^{\alpha}[W(t+s,1,1) - W(t,1,1)]f^h(s) ds\right| 
& \leq \frac{M}{h} \int_0^{\alpha}sf^h(s)ds \\ 
&\overset{h \to 0}{\to} M\left(2 - \int_{\alpha}^{\infty} s g(s) ds\right),
\end{align*}
and then taking the limit for $\alpha \to 0$.
\end{proof}

We can now apply this technical Lemma to our value function $V$ by checking that the assumptions of the
previous Lemma are satisfied by the value function $V$:

\begin{lemma}\label{lem:vderivatives}
Let $V: D \to \R$ be the value function for \eqref{V4}. Then,
\begin{enumerate}[(i)]
\item \begin{align}\label{V9a}
V_{b^*} (t,1,1):= \lim_{h \searrow 0} \frac{1}{h} [V(t,1,&1+h) - V(t,1,1)]=0 \\
& \text{ for } 0 < t < t_0 \nonumber
\end{align}
\item \begin{align}
V_{b}(t,1,1)&:= \lim_{h \searrow 0} \frac{1}{h} [V(t,1,1) - V(t,1-h,1]\label{p4a}\\
&\phantom{:}= -\int^\infty_0 [V(t+s, 1, 1) - V(t,1,1)]g(s)ds \nonumber
\end{align}
\end{enumerate}
\end{lemma}

\begin{proof}
$V$ is continuous by Lemma \ref{lem:valueprops} and Lipschitz-continuous in $b^*$ by definition. We also have $V(t,1,1) - \sqrt{t}$ is
decreasing by Lemma \ref{lem:valueprops}. Furthermore $V(t,B(t),B^*(t))$ is a martingale up to hitting $S$ by 
Lemma \ref{lem:valuemartingale}. Thus, setting $W=V$, $S_W = S$ and $t_W = t_0$, it remains to check the required uniform integrability condition. We have
\begin{align*}
0 \leq & V(u\wedge \tau^h, B(u \wedge \tau^h), B^*(u \wedge \tau^h)) \\
& \leq V(t,B(u \wedge \tau^h), B^*(u \wedge \tau^h)) + (\sqrt{u \wedge \tau^h} - \sqrt{t}) \\
& \leq V(t,0,1+h) + \sqrt{\tau^h} \leq V(t,0,1+h) + \frac{1}{4} + \tau^h.
\end{align*}
where the first estimate follows because the function
$u \mapsto V(u,b,b^*) - \sqrt{u}$ is decreasing in $u$. The second inequality is due to the
fact that $V$ is decreasing in $|b|$ and increasing in $b^*$ as well as $B^*(u \wedge \tau^h) \leq 1+h$.
Now, $\tau^h \leq \sigma^h$ and $\sigma^h$ has exponential moments and is therefore integrable which
yields the desired uniform integrability.

Observe that $\tau^h$ is smaller than the first hitting time of 
the non-stop region $NS$
no matter whether we condition on $(t,1,1)$ or $(t,1,1+h)$ which warrants the use of Lemma \ref{lem:valuemartingale}.
\end{proof}

The subsequent lemma shows that, for $t > t_0$, the behavior of $V_b$ and $V_{b^*}$ follows a different
pattern than the one given by Lemma \ref{lem:vderivatives}. We find that 
\begin{equation}\label{V10}
V_{b^*}(t,1,1)+V_b(t,1,1)=-C
\end{equation}
where we have to interpret this equation properly.

\begin{lemma}\label{3.2}
For $t >t_0$ we have
\begin{align}
\lim_{h \searrow 0} \frac{1}{h} [V(t,1+h, 1+h) - V(t, 1, 1)]=-C.
\end{align}
\end{lemma}

\begin{proof}
For $t >t_0$ we have 
\begin{align*}
V(t,b^*, b^*) = t^{\frac{1}{2}} - Cb^*,
\end{align*}
for $b^*$ in a neighbourhood of $1$.
\end{proof}

To abbreviate notation we shall sometimes denote by $V(t)$ the function $V(t,1,1)$ (recall that we keep $C \geq \widehat{C}$ and the corresponding $t_0 = t_0(C)$ fixed). We thus obtain the following integro-differential equation for $V(t).$

\begin{lemma}
The function $V(t)$ satisfies the following equations
\begin{align}
&2tV'(t) = V(t) + C, \quad &t > t_0,\label{V11}\\
&2tV'(t) = V(t) + \int^\infty_0 [V(t+s) - V(t)] g(s)ds, \quad &t < t_0.\label{V11a}
\end{align}
\end{lemma}

\begin{proof}
The first assertion is obvious, as we have
\begin{align}
V(t)=t^{\frac{1}{2}} - C, \quad \text{for} \quad t > t_0.
\end{align}

The second equation follows, at least formally, from \eqref{V7a}, \eqref{V9a} and \eqref{p4a}.

To justify \eqref{V11a} in a more pedantic way, note that for $a >1$ we obtain from \eqref{p3}
\begin{align*}
a V(t,1,1) - V(t,1,1)&= V(a^2t, a, a) - V(t,1,1)\\
&=(V(a^2t, a,a) - V(t,a,a))\\
&+ (V(t,a,a) - V(t,1,a))\\
&+ (V(t,1,a) - V(t,1,1)).
\end{align*}

Dividing by $a-1$ and letting $a$ decrease to $1,$ we deduce \eqref{V11a} from Lemma \ref{lem:vderivatives}.
\end{proof}

Let us discuss the behaviour of the function $V(t)$ at $t=t_0$. As observed in the previous section, $V(t)$ is continuous so that we must have ``continuous pasting'' at $t_0.$ It is the immediate reflex -- at least it was so for the present authors -- to expect {\it smooth pasting} of $V(t)$ at $t=t_0$ i.e. $\lim_{t \searrow t_0} V'(t)=\lim_{t \nearrow t_0} V'(t).$ By \eqref{V11} and \eqref{V11a} this would result in determining $t_0$ by equating $C$ with $\int^\infty_0(V(t_0 + s) - V(t_0)) g(s)ds.$ To our big surprise this turned out \emph{not} to be the case; after some time of reconsidering we had to conclude that there is little reason why the smooth pasting principle should prevail in the present context. Here is one intuitive reason: for a fixed number $t_0>0$ we have that for almost all trajectories of a Brownian motion $B=(B(t))_{t \geq 0}$, starting at $B(0)=0$, there is no $t>0$ such that the two equalities $|B(t)|=B^*(t)=(\frac{t}{t_0})^{\frac{1}{2}}$ are simultaneously verified. By Lemma \ref{lem:stopreg} we conclude that a  discontinuity of the derivatives of $V(t, B(t), B^*(t))$ can only take place where these two equations are simultaneously satisfied. Roughly speaking: the Brownian motion $B$ ``{\it does not see}'' a kink of the function $V(t)$ at $t=t_0.$ 

As a matter of fact, this natural example of a case of \emph{non-smooth pasting} in the case of continuous martingales seems to us a remarkable feature of the present paper. The literature on \emph{non-smooth pasting} is generally revolving around non-continuous processes. Some examples of non-smooth pasting for processes with jumps can be found in \cite{AlKy05}, \cite{AsAvPi04}, \cite{BoLe02}, \cite{DaHo04} and \cite{PeSh00}.

\section{The Integro-Differential-Equation}\label{sec:oide}

Fix the parameters $C>0$ and $t_0 >0$. We consider the ordinary integro-differential equation for the function $U=U^{C,t_0}$ 
\begin{align}
U(t)&=t^{\frac{1}{2}}-C, \qquad &t \geq t_0,\label{V14}\\
2tU'(t)&= U(t) + \int^\infty_0 [U(t+s) - U(t)] g(s)ds, \quad &0 < t\leq t_0,\label{V14a}
\end{align}
where $g$ is given by \eqref{form:derdens}.

Here the fixed behaviour \eqref{V14} of $U(t)$, for $t \geq t_0$, is considered as the initial condition, and subsequently the OIDE (ordinary integro-differential equation) \eqref{V14a} is solved by letting $t$ decrease from $t_0$ to $0$. For $t=t_0$, the derivative $U'(t_0)$ in \eqref{V14a} is understood as the {\it left} limit of $U'(t)$ when $t$ increases to $t_0.$ 

It is standard to verify that, for $C>0, t_0>0,$ and $\epsilon >0$ the solution $U^{C,t_0}$ of \eqref{V14} is well-defined for $t \in [\epsilon, \infty)$ and depends smoothly on the parameters $C$ and $t_0.$ On the other hand, the $2t$ term on the left hand side of \eqref{V14a} indicates that only for special cases of $C$ and $t_0$ this solution can be extended to a continuous and finitely valued function $U(t),$ defined for all $t \in \, [0,\infty)$.

The evidence resulting from our numerical analysis of the solutions $(U^{C,t_0}(t))_{t>0}$, in dependence of $C$ and $t_0$, can be resumed as follows:

\begin{numerics}\label{ne}
$\quad$

\begin{enumerate}[(i)] 
\item There is a smallest number $\widecheck{C} >0$ as well as a unique number $\widecheck{t}_0 >0$ such that the OIDE \eqref{V14a} admits a solution $U^{\widecheck{C},\widecheck{t}_0}(t)$ which has a finite limit $\lim_{t \to 0} U^{\widecheck{C},\widecheck{t}_0}(t)$. This solution is monotone increasing and $U^{\widecheck{C}, \widecheck{t}_0}(t) \geq t^{\frac{1}{2}} - \widecheck{C}.$

\item For $0 < C < \widecheck{C}$ there does not exist $t_0 >0$ such that $U(t)=U^{C,t_0}(t)$ remains bounded from below as $t \searrow 0.$

\item For $C > \widecheck{C}$ there are precisely two values $t_1 = t_1(C), t_2=t_2 (C)$ depending on $C$ in a continuous and one-to-one way such that $U^{C,t_i}(t)$ has a finite limit, as $t \searrow 0,$ for $i=1,2$. These solutions satisfy $U^{C,t_i}(t) \geq t^{\frac{1}{2}}-C$. \\For $t_0 \in (t_1,t_2)$, the solutions of the OIDE \eqref{V14a} tend to $+\infty$, for $t \searrow 0$, while, for $t_0 \notin [t_1,t_2]$, the solutions tend to $-\infty$, for $t \searrow 0$. We have $t_1(C) < \widecheck{t}_0 < t_2(C)$ and $\lim_{C \to \widecheck{C}} t_1(C) = \lim_{C \to \widecheck{C}} t_2(C)=\widecheck{t}_0$.

The functions $U^{C, t_1}$ and $U^{C,t_2}$ are monotone increasing and 
\begin{equation}\label{form:strongineq}
U^{C,t_1}(t) < U^{\widecheck{C}, \widecheck{t}_0}(t)\,, \quad \text{for all} \quad t \in [0, \infty[.
\end{equation}
\end{enumerate}

Finally, we find the numerical values 
\begin{align}\label{Q1}
\widecheck{C}\approx 1.27267\dots \quad \text{and} \quad \widecheck{t}_0 \approx 0.9036\dots
\end{align}
\end{numerics}

We have not been able to provide a mathematically rigorous proof of the above assertions and only rely on the numerical evidence (which is based on Euler-type simulations in Python with variable step sizes). We therefore consider the above statements rather as hypotheses underlying our subsequent results and we shall carefully point out in the subsequent statements where we rely on this evidence.

For example, for $C=1.25 < \widecheck{C}$ which is case (ii) above we illustrate the situation by Figure \ref{fig:sub}.

\vskip10pt
\begin{figure}[ht]
\includegraphics[width=0.7\textwidth]{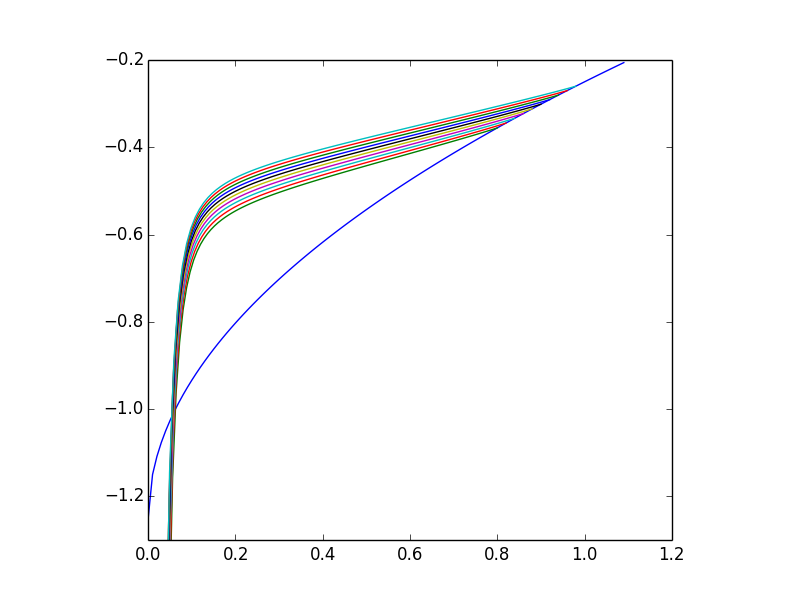}
\caption{The subcritical case $C < \widecheck{C}$: Numerical solutions for $C=1.25$ and various values for
the pasting position $t_0$ in the interval $[0.8,1]$. The graph underneath is $t^{1/2}-C$ through which all solutions cut in the subcritical case when they get close to $0$.}
\label{fig:sub}
\end{figure}

For $C = 1.274 > \widecheck{C}$, which is case (iii) above, we find $t_1 \approx 0.85\dots$ and $t_2 \approx 0.95\dots,$ as illustrated in Figure \ref{fig:super}.

\vskip10pt
\begin{figure}[ht]
\includegraphics[width=0.7\textwidth]{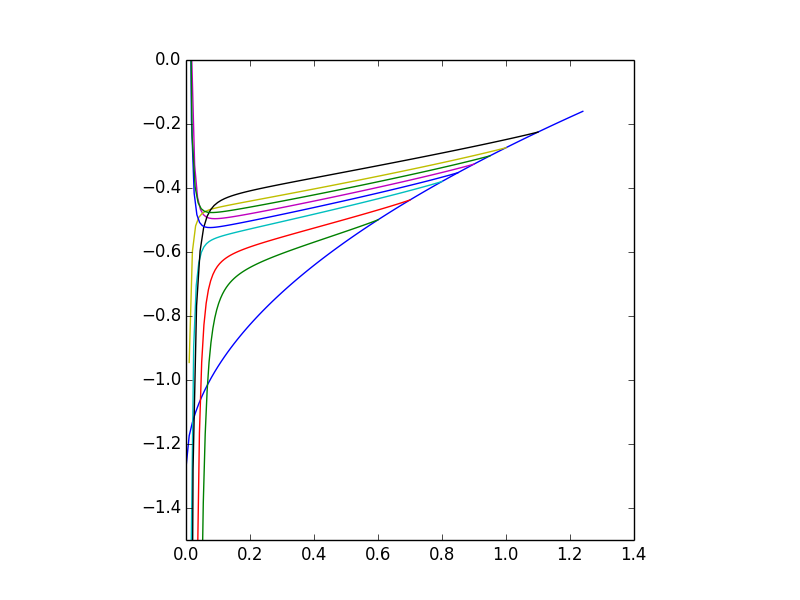}
\caption{The supercritical case $C > \widecheck{C}$: Numerical solutions for $C=1.274$ and and various values
for the pasting position $t_0$ in the interval $[0.6,1.2]$. We find that while solutions for $t_0$ in an interval $(t_1,t_2)$ with $t_1 \approx 0.85$ and $t_2 \approx 0.95$ stay above the graph of $t^{1/2}-C$ (in fact, they tend to $+\infty$, as $t \to 0$), the solutions fall to $-\infty$ for $t \to 0$ if $t_0 \notin [t_1,t_2]$.
At the transition of these two regimes lie the bounded solutions $U^{C,t_1}$ and $U^{C,t_2}$ (which are not explicitly
displayed in the above figure but are squeezed between the neighbouring solutions $U^{C,t}$).}
\label{fig:super}
\end{figure}
\vskip10pt

When $C$ decreases to the critical value $\widecheck{C} \approx 1.27267\dots$ the numerics suggest that the length $t_2 - t_1$ of the intervals $(t_1, t_2)$ decreases to zero and that these intervals shrink to a single point $\widecheck{t}_0 \in \, ]0, \infty[$ for which we find $\widecheck{t}_0 \approx 0.9036\dots$. It is convincing from the numerics that the limiting solution $U^{\widecheck{C}, \widecheck{t}_0}(t)$ then is well-defined for all $t \geq 0$ by letting $U^{\widecheck{C}, \widecheck{t}_0}(0):=\lim_{t \searrow 0} U^{\widecheck{C}, \widecheck{t}_0}(t).$ This function $U^{\widecheck{C}, \widecheck{t}_0}(t)$ is monotone increasing and such that $U^{\widecheck{C}, \widecheck{t}_0}(t)$ lies above the function $U^{\widecheck{C},0}(t)=t^{\frac{1}{2}} - \widecheck{C}.$ Clearly we expect that $U^{\widecheck{C}, \widecheck{t}_0}$ must be the ``right'' solution, which may
be identified with the value function $V$ defined in \eqref{V4} for the optimal constant $\widehat{C}=\widecheck{C}$ and, in particular, that $\widecheck{C} \approx 1.27267\dots$ equals the optimal constant $\widehat{C}$ in the Burkholder-Davis-Gundy inequality \eqref{V1}. We shall subsequently deduce this result more formally.

\section{Identifying the Value-Function}\label{sec:synth}

Admitting the Numerical Evidence \ref{ne} we shall show that the function $U^{\widecheck{C}, \widecheck{t}_0}(t)$, obtained above from the analysis of the OIDE \eqref{V14a}, indeed determines the value function 
$V(t,b,b^*)$ as defined in \eqref{V4} for the constant $\widehat{C}=\widecheck{C}$ and that this constant is indeed the optimal Burkholder-Davis-Gundy constant in inequality \eqref{V1}.

Starting from a solution $U(t)=U^{C,t_0}(t)$ of the OIDE \eqref{V14a} for parameters $C>0$ and $t_0 >0$ such that $U(t)$ extends continuously to a finite value $U(0)$ we may extend this solution (by slight abuse of notation) to a function $U(t,b,b^*)$, defined on $D$, by first letting 
\begin{align}\label{p12}
U(t,b,1):= \int^{\infty}_0 U(t+s) f^{1-|b|} (s)ds, \quad \text{for} \quad 0 \leq |b|< 1,
\end{align}
where $f^h(s)$ is defined in Section 3. For general $(t,b,b^*) \in D$ we use \eqref{p3} to define 
\begin{align}\label{p12a}
U(t,b,b^*)= b^* U\Big(\frac{t}{(b^*)^2}, \frac{b}{b^*},1\Big).
\end{align}

For later reference, we note that $U(t) \geq \sqrt{t} - C$ implies
\begin{align} \label{eqn:lowerestimate}
U(t,b,b^*) \geq \sqrt{t} - Cb^*, \text{for all } (t,b,b^*) \in D
\end{align}

\begin{lemma}\label{5.1}
Fix $C >0$ and $t_0 > 0$ such that $U(t)$ extends continuously to a finite $U(0)$ and admit the Numerical Evidence \ref{ne} (i) and (iii). Let $(B(u))_{u \geq t}$ be a Brownian motion starting at some time $t>0$ at $B(t)=b$ and $B^*(t)=b^*$. The process $U (u,B(u), B^*(u))_{u \geq t}$ is then a local super-martingale. It is a local martingale up to entering the stopping area $S := \{(t,b,b^*) : |b|=b^*, t/(b^*)^2 \geq t_0\}$.
\end{lemma}

For the proof we need the following Lemma to justify the use of Ito's formula for a function that is not
smooth everywhere but where the Brownian Motion hardly ever touches the set where it is not differentiable.

\begin{lemma}\label{lem:almostmartingale}
Let $W: D \to \R$ be a continuous function and $t_W > 0$ such that 
\begin{enumerate}[(a)]
\item the derivatives $W_t$ and $W_{bb}$ exist and are continuous on the interior of $D$, and
\item $W_t + \frac{1}{2}W_{bb} = 0$,
\item $W_{b^*} \leq 0$, and
\item $W_{b^*}(t,b,|b|) = 0$ for $\frac{t}{b^2} < t_W$.
\end{enumerate}
Define $S_W := \{(t,b,|b|) \in D: \frac{t}{b^2} \geq t_W\}$. For a standard Brownian Motion $B(t)$ let
$\tau_W$ be the first hitting time of $S$ for $(t,B(t),B^*(t))$. Then,
\begin{enumerate}[(i)]
\item $X(t) := W(t,B(t),B^*(t))$ is a local supermartingale, and
\item $X(t \wedge \tau_W)$ is a local martingale.
\end{enumerate}
\end{lemma}

\begin{proof}
This follows at least formally from the assumptions and Ito's formula
\begin{align}
d X(t) =(W_t + \tfrac{1}{2} W_{bb}) dt + W_b dB(t)+ W_{b^*}dB^*(t).
\end{align}
To address this in a more formal way, let $\epsilon >0$ and define the stopping times $(\rho^\epsilon_n)^\infty_{n=0}$ by $\rho^\epsilon_0=0$ and
$$\rho^\epsilon_n=\inf\{t:t \geq \rho^\epsilon_{n-1} + \epsilon \quad \text{and} \quad |B(t)|=B^*(t)\}.$$

We also denote by $A^\epsilon$ the union $\bigcup^\infty_{n=0} \llbracket \rho^\epsilon_n, \rho^\epsilon_n + \epsilon \rrbracket$ which is a predictable subset of $\Omega \times \mathbb{R}_+.$  Denoting by $A^0=\bigcap^\infty_{n=0} A^{{\frac{1}{n}}},$ the set $A^0$ simply equals $\{|B(t)|=B^*(t)\}.$ Fixing $T>0,$ the Lebesgue-measure of $\{\omega\} \times [0,T] \cap A^\epsilon$ tends to zero, for almost all $\omega \in \Omega.$

Fix a bounded stopping time $\tau$ such that $(B(t))_{0 \leq t \leq \tau}$ remains bounded. It follows that $W_t$ as well as $W_{bb}$ also remain bounded on $A^\epsilon \cap \llbracket0,\tau \rrbracket$ so that
$$X^\epsilon(t):= \int^t_0 \mathbbm{1}_{B^\epsilon} d X(s)$$
is a martingale and $B^\epsilon=(\Omega \times \mathbb{R}_+) \backslash A^\epsilon$ is the complement of $A^\epsilon.$ Indeed, it suffices to reason on the stochastic intervals $\llbracket \rho^\epsilon_{n-1} + \epsilon, \rho^\epsilon_n \rrbracket$ and to observe that $B^*(t)$ remains constant on these intervals.

Turning to the remaining part 
$$Y^\epsilon(t) := \int^t_0 \mathbbm{1}_{A^\epsilon} dX(s) = X(t)-X^\epsilon (t)$$
we shall show that along a sequence these processes tend almost surely to the non-increasing process
\begin{align*}
Y^0(t)&=\int^t_0 \mathbbm{1}_{A^0} dX(s)\\
&=\int^t_0\mathbbm{1}_{A^0} W_{b^*}(t,B(t), B^*(t)) d B^*(t).
\end{align*}
Indeed, the dominated convergence theorem for Ito-Integrals yields convergence in probability 
and thus subsequence convergence almost surely. Fixing this sequence of $\varepsilon$'s we can take the process to the
appropriate limit.
\end{proof}

\begin{proof}[Proof of Lemma \ref{5.1}]
It follows from definition \eqref{p12} that, for $(t,b,b^*) \in D$ such that $t > 0$ and $0 \leq |b|<b^*,$ the heat equation
\begin{align*}
U_t+ \frac{1}{2} U_{bb}=0
\end{align*}
is satisfied.

On the boundary, for $t < t_0$, we can apply the definition of $U$ to obtain
\begin{align*}
U_b(t,1,1) + U_{b^*}(t,1,1) &:= \lim_{h \searrow 0} \frac{1}{h} [U(t,1+h,1+h) - U(t,1,1)] \\
&=\lim_{h \searrow 0} \frac{1}{h} [U(t/(1+h)^2) - U(t)] + U(t/(1+h)^2) \\
&= -2tU'(t) + U(t) \\
&= -\int_0^\infty [U(t+s)-U(t)]g(s)ds
\end{align*}
The last equality is exactly the OIDE \eqref{V14a}.  
Applying Lemma \ref{lem:valuederivatives} (ii) one obtains that the last expression is equal to 
\[U_b(t,1,1) := -\lim_{h \searrow 0}  \frac{1}{h} [U(t,1-h,1) - U(t,1,1)], \] 
where all the assumptions of this Lemma are easily checked.
Clearly the left and right derivatives of $U_b$ agree. It follows that $U_{b^*}(t,1,1) = 0$ and more generally that 
$U_{b^*}(t,b,|b|) = 0$ for $t/b^2 < t_0$. For $t \geq t_0$ we can derive as in Lemma \ref{3.2} that
\[U_{b^*}(t,1,1) = -U_b(t,1,1) - C = \int_0^{\infty}[(t+s)^{\frac{1}{2}} - t^{\frac{1}{2}}]g(s)ds-C.\]
This expression is monotone decreasing in $t$, and is necessarily non-positive at $t_0$
so that $U(t,1,1) \geq t^{\frac{1}{2}}-C$ holds. We conclude that, for arbitrary $t > 0$, 
$U_{b^*} \leq 0$.

Having established that $U_{b^*} \leq 0$ and $U_{b^*}(t,b,|b|) = 0$ for $\frac{t}{b^2} < t_0$ we may derive,
at least formally, the assertion of the present lemma from \eqref{P1} and Ito's formula as in \eqref{p5}.

Now, we can conclude using Lemma \ref{lem:almostmartingale}
\end{proof}

Let us now observe the following relations between value functions to the optimal stopping problem
and solutions to the OIDE \eqref{V14a}. 

\begin{lemma}\label{lem:valueisoide}
Let $V^C$ be the value function as defined in \eqref{V4} for a constant $C > 0$ that satisfies 
the inequality \eqref{V1}. Take $t_0 = t_0(C) \in (0,\infty)$ to be the corresponding point separating
$S$ from $NS$ (see \eqref{p2.1}). Then $V^C(t):= V^C(t,1,1)$ satisfies the OIDE \eqref{V14a} for this choice
of $C$ and $t_0$.
\end{lemma}

\begin{proof}
For $t\geq t_0$, we have $V^C(t)=t^{\frac{1}{2}}-C$.

For $t < t_0$ denote by $\tau^h=\sigma^h \wedge \tau$ the stopping time as in \eqref{4.12} above, conditionally on $(t,1,1).$ Note that $(V^C(u, B(u), B^*(u)))_{t \leq u \leq \tau^h}$ then is a \emph{uniformly integrable} martingale. To see this,
we make a distinction for $u \leq \nu$ and $u > \nu$ where $\nu$ is the stopping time 
$\nu := \inf \{s : B^*(s)^2 t_0 \leq s\}$:
\begin{enumerate}
\item $u \leq \nu^h := \nu \wedge \tau^h$: The domain of $((u,B(u),B^*(u))$ is bounded by $|B(u)| \leq B^*(u) \leq 1+h$ and
$u \leq \nu^h \leq B^*(\nu^h)^2t_0 \leq (1+h)^2t_0$. Clearly $V^C$ is bounded on this domain.
\item $u > \nu^h$: Here the properties of $V^C$ given in Lemma \ref{lem:valuemartingale} allow us to see that one can rewrite
the process as
\[V^C(u,B(u),B^*(u)) = \E^{(u,B(u),B^*(u))}[\sigma^{\frac{1}{2}}] - CB^*(u)\]
where $\sigma$ is again the first time where $B(u)=B^*(u)$ holds. Note that this holds because $B^*(v)$ is now constant
for $u \leq v \leq \sigma$, and it is clearly the definition of a uniformly integrable martingale, 
provided the conditional expectation
is well defined, which it is by the estimate 
\[\E^{(u,B(u),B^*(u))}[\sigma^{\frac{1}{2}}] \leq u^{\frac{1}{2}}+B^*(u).\]
\end{enumerate} 

Hence the formula
\begin{align*}
V^C(t,1,1)=\E^{(t,1,1)}[V^C(\tau^h, B(\tau^h), B^*(\tau^h)]
\end{align*}
is justified.
This was the only assumption of Lemma \ref{lem:valuederivatives} which is not immediate
from the definition of $V^C$. Now the claim follows from Lemma \ref{lem:valuederivatives}.
\end{proof}

\begin{proposition} \label{5.2}
Admitting the Numerical Evidence \ref{ne} (i) and (ii), the constant $\widecheck{C}$ obtained in \eqref{Q1} equals the optimal constant $\widehat{C}$ for \eqref{V1}, and the function $U^{\widecheck{C}, \widecheck{t}_0} (t,b,b^*)$ obtained in \eqref{p12} and \eqref{p12a} equals the value function $V(t,b,b^*)$ as defined in \eqref{V4} for the constant $\widehat{C}=\widecheck{C}.$

The value $\widehat{t}_0$ associated to $\widehat{C}$ by Lemma \ref{lem:stopreg} equals the constant $\widecheck{t}_0$ in \eqref{Q1}.
\end{proposition}

\begin{proof}
To show $\widehat{C} \geq \widecheck{C},$ suppose that $C>0$ is a constant satisfying the Burkholder-Davis-Gundy inequality \eqref{V1}, i.e.~suppose that $C \geq \widehat{C}$. Then by Lemma \ref{lem:valueisoide} we have that $V^C(t):=V^C(t,1,1)$ satifies the OIDE \eqref{V14a} for this choice of $C$ and the corresponding $t_0(C)$ separating $S$ from $NS$. 

As $V^C(t)$ is increasing in $t$ and satisfies $V^C(0) \geq -C$ we conclude from the {\it Numerical Evidence \ref{ne}} (i) and (ii) that $C \geq \widecheck{C}$. This yields $\widehat{C} \geq \widecheck{C}.$

To show conversely that $\widehat{C} \leq \widecheck{C}$ consider the function $U^{\widecheck{C}, \widecheck{t}_0} (t,b,b^*).$ By Lemma \ref{5.1} the process $(U^{\widecheck{C}, \widecheck{t}_0} (t,B(t),B^*(t))_{t \geq 0}$ is a local supermartingale. Hence we have, conditionally on $(t,b,b^*) \in D$ and for each bounded stopping time $\tau \geq t$ and localizing sequence $(\tau_n)_{n=1}^{\infty}$.

\begin{align*}
U^{\widecheck{C}, \widecheck{t}_0} (t,b,b^*) &\geq \mathbb{E}^{(t,b,b^*)} [U^{\widecheck{C}, \widecheck{t}_0} (\tau \wedge \tau_n, B( \tau \wedge \tau_n), B^*(\tau \wedge \tau_n) )]\\
&\geq \mathbb{E}^{(t,b,b^*)} [(\tau \wedge \tau_n)^{\frac{1}{2}} - \widecheck{C} B^* (\tau \wedge \tau_n)]
\end{align*}
where the second inequality derives from \eqref{eqn:lowerestimate}.
In the limit for $n \to \infty$ this yields $U^{\widecheck{C}, \widecheck{t}_0} (t,b,b^*) \geq \mathbb{E}^{(t,b,b^*)} [\tau^{\frac{1}{2}} - \widecheck{C} B^* (\tau)]$. Hence $U^{\widecheck{C}, \widecheck{t}_0} (t,b,b^*)$ dominates the value function $V^{\widecheck{C}}(t,b,b^*)$ as defined in \eqref{V4} for the constant $\widecheck{C}.$ This shows $\widecheck{C} \geq \widehat{C}$ as well as $U^{\widecheck{C}, \widecheck{t}_0} \geq V^{\widecheck{C}} = V^{\widehat{C}}.$
\end{proof}

We can finally summarize these results to proof the main theorem.

\begin{proof}[Proof of Theorem \ref{thm:mainoide}]
Suppose that $C=\widehat{C}$ is the optimal constant for \eqref{V1} and let $\widehat{t}_0 \in \,(0,\infty)$ the corresponding critical value given by Lemma \ref{lem:stopreg}. Then $V=V(t)$ satisfies the OIDE \eqref{V14} and this solution is increasing in $t$ and satisfies $V(0) \geq -C.$

As shown in the previous section there is a minimal $C$ allowing for such a solution, for an appropriately chosen $t_0 \in \,(0,\infty).$ This value of $C$ therefore must coincide with the optimal value $\widehat{C}$ for the Burkholder-Davis-Gundy inequality 
\eqref{V1}.

Conversely if $C$ is chosen such that for some $t_0 \in (0,\infty)$ the OIDE \eqref{V14},\eqref{V14a} has a solution on $[0,T]$ then
the numerical evidence \ref{ne} (ii) gives that $C \geq \bar{C} = \hat{C}$ and so \eqref{V1} holds.
\end{proof}

\begin{remark}
It is interesting to consider,  for a fixed constant $C > \widehat{C}$, the relation between the value-function $V^C(t,b,b^*)$ defined in \eqref{V4} and the corresponding solutions of the OIDE \eqref{V14a}. In this case the numerical evidence \ref{ne} (iii) indicates that there are two bounded solutions $U^{C,t_1}(t)$ and $U^{C,t_2}(t).$ Which of the two is the ``good one'', i.e.~which one equals the value function $V^C(t,1,1)$?

To answer this question, first note that, for $C>\widehat{C}$, we clearly have the monotonicity relation $V^C(t,b,b^*)\leq V^{\widehat{C}}(t,b,b^*)$. It is also easy to see that $t_0(C) < t_0(\widehat{C})=\widehat{t}_0 = \widecheck{t}_0$, where $t_0(C)$ is associated to the value function $V^C(t,b,b^*)$ via Lemma \ref{lem:stopreg}. In other words, the stopping region $S$ for the function $t^{\frac{1}{2}} - CB^*(t)$ in \eqref{V4} is bigger than the stopping region $S$ for the function $t^{\frac{1}{2}} - \widehat{C}B^*(t).$

It follows from the numerical evidence that the value $t_1(C)$ for which we have $t_1(C) < \widehat{t}_0$ is the only candidate  for the ``good'' solution while for $t_2(C)$ for which we have $t_2(C) > \widehat{t}_0,$ we cannot have $U^{C,t_2(C)}(t)=V^C(t,1,1).$ 
We can conclude from Lemma \ref{lem:valueisoide} that the value function $V^C(t,1,1)$ indeed equals the solution
$U^{C,t_1(C)}(t)$ of the OIDE \eqref{V14a}.

The fact that $U^{C,t_2(C)}(t)$ cannot be the ``good'' solution has the following consequence which is interesting in its own right (compare \cite{Sh67}).
\end{remark}

\begin{proposition} \label{prop:infexp}
Admitting the Numerical Evidence \ref{ne} (iii) we have that, for $t_2 >\widehat{t}_0,$ the stopping time
\begin{align}\label{R2a}
\rho=\inf\Big\{s \geq 1:\frac{s}{(B^*(s))^2} \geq t_2\Big\}
\end{align}
satisfies
\begin{align}\label{R2aa}
\mathbb{E}[\rho^{\frac{1}{2}}]=\infty.
\end{align}
\end{proposition}

\begin{proof}
Define the stopping time $\tau$ by 
\begin{align} \label{form:stopentry}
\tau := \inf \left\{s \geq 1: \frac{s}{B^*(s))^2} \geq t_2 \text{ and } |B(s)|=B^*(s) \right\}.
\end{align}
Clearly $\tau \geq \rho$, as we may equivalently define
\begin{align} \label{form:stopentryb}
\tau := \inf \{s \geq \rho : |B(s)| = B^*(s)\}.
\end{align}
We claim that $\E[\rho^{\frac{1}{2}}] < \infty$ if and only if $\E[\tau^{\frac{1}{2}}]< \infty$. Indeed, it follows from
\eqref{form:stopentryb} that the law of $\tau-\rho$, conditionally on $(\rho,B(\rho),B^*(\rho))$, is that of the first hitting time
$\sigma$ of the level $B^*(\rho)$ by the absolute value of a Brownian motion $(W(u))_{u \geq 0}$ starting at
$W(0)=B(\rho)$. We may (very crudely) estimate $\E[\sigma] \leq B^*(\rho)^2$.

Noting that at time $\rho$ we have $B^*(\rho)^2 \leq \frac{\rho}{t_2}$ we may estimate
\begin{align*}
\E[\tau-\rho | \rho,B(\rho),B^*(\rho)] \leq \frac{\rho}{t_2}.
\end{align*}
Hence we obtain
\begin{align*}
\E[\tau^{\frac{1}{2}}] &= \E[\rho^{\frac{1}{2}}] + \E[\E[\tau^{\frac{1}{2}}-\rho^{\frac{1}{2}}|\rho,B(\rho),B^*(\rho)]]\\
& \leq \E[\rho^{\frac{1}{2}}] + \E[\E[(\tau-\rho)^{\frac{1}{2}}|\rho,B(\rho),B^*(\rho)]]\\
&\leq \E[\rho^{\frac{1}{2}}] + \E[\E[\tau-\rho|\rho,B(\rho),B^*(\rho)]^{\frac{1}{2}}]\\
&\leq \left(1+\frac{1}{\sqrt{t_2}}\right)\E[\rho^{\frac{1}{2}}],
\end{align*}
which readily shows that $\E[\rho^{\frac{1}{2}}] < \infty$. This implies that $\E[\tau^{\frac{1}{2}}] < \infty$.

So let us suppose that $\E[\tau^{\frac{1}{2}}] < \infty$ and work towards a contradiction.

Define the stopping region $S(t_2)$ relative to $t_2$ as 
\[S(t_2) = \left\{(t,b,|b|) \in D : \frac{t}{b^2} \geq t_2\right\}.\]
and the corresponding non-stopping region by $NS(t_2) = D\backslash S(t_2)$.

We condition on some fixed $(1,b,b^*) \in NS(t_2)$. Note that $\tau$ is the first time when $(t,B(t),B^*(t))_{t \geq 1}$
leaves $NS(t_2)$.

Admitting the Numerical Evidence \ref{ne} (iii), associate to $t_2 >\widehat{t}_0$ the constant $C >\widehat{C}$ such that $U^{C,t_2}(t)$ is a solution of the OIDE \eqref{V14a} which remains bounded as $t \searrow0.$ We write $U^{C,t_2}(t,b,b^*)$ for its extension defined in \eqref{p12a}. In contrast, we denote by $V^C(t,b,b^*)$ the value function as defined in \eqref{V4} for the constant $C$.

The process  $(U^{C,t_2}(t,B(t),B^*(t)))_{1 \leq t \leq \tau}$ is a local martingale by Lemma \ref{5.1}, where the present
$t_2$ corresponds to $t_0$ in the statement of this lemma.

In addition we show that this local martingale is a uniformly integrable martingale up to time $\tau$, 
i.e. the family of random variables
$U^{C,t_2}(\sigma,B(\sigma),B^*(\sigma))$, where $\sigma$ ranges in the stopping times $1\leq \sigma \leq \tau$,
is uniformly integrable. Recall the scaling relation
\[U^{C,t_2}(\sigma,B(\sigma),B^*(\sigma)) = B^*(\sigma)U^{C,t_2}\left(\frac{\sigma}{B^*(\sigma)^2},\frac{B(\sigma)}{B^*(\sigma))},1\right),\]
and note that $\frac{\sigma}{B^*(\sigma)^2}$ remains in the interval $[0,t_2]$ so that, by compactness,
the term $|U^{C,t_2}(\frac{\sigma}{B^*(\sigma)^2},\frac{B(\sigma)}{B^*(\sigma))},1)|$ remains bounded by some constant
$M > 0$. Therefore
\[|U^{C,t_2}(\sigma,B(\sigma),B^*(\sigma))| \leq MB^*(\sigma) \leq MB^*(\tau).\]
If $\mathbb{E}[\tau^{\frac{1}{2}}] < \infty$ we infer from the Burkholder-Davis-Gundy inequality (this time the reverse inequality to \eqref{V1}) that the random variable $B^*(\tau)$ is integrable. Hence the family of random variables $U^{C,t_2}(\sigma, B(\sigma), B^*(\sigma))$ is dominated by the integrable random variable $M B^*(\tau)$ which shows that the local martingale $U^{C,t_2}(t,B(t),B^*(t))_{1 \leq t \leq \tau}$ is of class D and is thus a uniformly integrable martingale.

Hence, conditionally on $(1,b,b^*) \in NS(t_2)$ we obtain
\begin{align}\label{form:condef}
U^{C,t_2}(1,b,b^*) = \E^{(1,b,b^*)}[U^{C,t_2}(\tau,B(\tau),B^*(\tau))].
\end{align}

We now pass to the process $(V^C(t,B(t),B^*(t))_{t \geq 1}$ again conditionally on $(1,b,b^*) \in NS(t_2)$. By Lemma \ref{lem:valuemartingale} we know that this process is a supermartingale. Repeating the above argument, we obtain that this supermartingale is uniformly integrable up to time $\tau$. Hence
\begin{align}\label{form:subdef}
V^C(1,b,b^*) \geq \E^{(1,b,b^*)}[V^C(\tau,B(\tau),B^*(\tau))].
\end{align}

Noting that at time $\tau$ we arrived in the stopping region $S(t_2)$ we obtain
\[U^{C,t_2}(\tau,B(\tau),B^*(\tau)) = V^C(\tau,B(\tau),B^*(\tau)) = \tau^{\frac{1}{2}} - CB^*(\tau).\]
Hence \eqref{form:condef} and \eqref{form:subdef} yield
\[V^C(1,b,b^*) \geq U^{C,t_2}(1,b,b^*),\]
for all $(1,b,b^*) \in D$. As we have seen that $V^C = U^{C,t_1(C)} \leq U^{C,t_2(C)}$, and $U^{C,t_1(C)}$ is
not equal to $U^{C,t_2(C)}$, we arrive at the desired contradiction.
\end{proof}

The above result is complemented by the following estimate in the reverse direction.

\begin{proposition}\label{prop:finexp}
Admitting the Numerical Evidence \ref{ne} (iii), we have, for $t_1 < \widehat{t}_0$, that the stopping time
\begin{align}\label{01}
\rho=\inf\Big\{s \geq 1:\frac{s}{(B^*(s))^2} \geq t_1\Big\}
\end{align}
satisfies
\begin{align}\label{01a}
\mathbb{E}[\rho^{\frac{1}{2}}] < \infty.
\end{align}
\end{proposition}

\begin{proof}
Similarly as in the proof of the previous proposition, we define
\[\tau := \inf \{s \geq 1: \frac{s}{(B^*(s))^2} \geq t_1 \text{ and } |B(s)|=B^*(s) \}.\]
We shall show that
\begin{equation}\label{form:finite}
\E[\tau^{\frac{1}{2}}] < \infty
\end{equation}
which will imply \eqref{01a}.

We condition on $(1,b,b^*)$ in the non-stop region $NS(t_1)$ defined as in the preceding proof, and will show
\begin{equation}\label{form:linbound}
\E^{(1,b,b^*)}[\tau^{\frac{1}{2}}] \leq Kb^*
\end{equation}
for some constant $K > 0$, which will imply \eqref{form:finite} by integrating
over the values $B(1)=b$ and $B^*(1)=b^*$.

We associate to $t_1 < \widehat{t}_0$ the corresponding $C>\widehat{C}$ such that the solution $U^{C,t_1}(t)$ of the OIDE \eqref{V14a} remains bounded (Numerical Evidence \ref{ne} (iii)).

Using the Numerical Evidence \ref{ne} (iii) there is some $\alpha >0$ such that the solutions $U^{C,t_1}(t)$ of the OIDE \eqref{V14a} and the value function $V^{\widehat{C}}(t) = U^{\widecheck{C},\widecheck{t}_0}(t)$ for the optimal constant $\widehat{C}=\widecheck{C}$ are separated by some $\alpha >0$, i.e.
\begin{align}\label{02a}
V^{\widehat{C}}(t) \geq U^{C,t_1}(t) + \alpha, \quad \text{for all} \quad t \geq 0.
\end{align}
Indeed, for $t \geq \widehat{t}_0$ we have $U^{C,t_1} = t^{\frac{1}{2}}-C$ and $U^{\widehat{C},\widehat{t}_0}(t)=t^{\frac{1}{2}}-\widehat{C}$. For $t \in [0,\widehat{t}_0]$ we have $U^{\widehat{C},\widehat{t}_0}(t) < U^{C,t_0}(t)$ by \eqref{form:strongineq}
so that by compactness we obtain a separating constant $\alpha>0$.

More generally, we obtain from \eqref{p12}
\begin{align}\label{02aa}
V^{\widehat{C}}(t,b,1) \geq U^{C,t_1}(t,b,1) + \alpha, \quad \text{for all} \quad (t,b,1) \in D.
\end{align}

Similarly as in the above proof we consider, conditionally on $(1,b,b^*)$, the processes 
\[(U^{C,t_1}(u,B(u),B^*(u)))_{1 \leq u \leq \tau} \text{ and } (V^{\widehat{C}}(u,B(u),B^*(u)))_{1 \leq u \leq \tau}.\]

Both are local martingales up to time $\tau$. Let $(\tau_n)_{n=1}^{\infty}$ be a sequence of localizing, bounded stopping times, $\tau_n \geq 1$, increasing to $\tau$.

\begin{align*}
V^{\widehat{C}}(1,&b,b^*)=\E^{(1,b,b^*)}[V^{\widehat{C}}(\tau_n, B(\tau_n), B^*(\tau_n)] \\
&= \E^{(1,b,b^*)}\left[ B^*(\tau_n) V^{\widehat{C}}\left(\tfrac{\tau_n}{B^*(\tau_n)^2}, \tfrac{B(\tau_n)}{B^*(\tau_n)},1\right)\right] \\
&\geq \E^{(1,b,b^*)}\left[ B^*(\tau_n) U^{C,t_1}\left(\tfrac{\tau_n}{B^*(\tau_n)^2}, \tfrac{B(\tau_n)}{B^*(\tau_n)},1\right)\right]
+\alpha\E^{(1,b,b^*)}[ B^*(\tau_n)] \\
&= U^{C,t_1}(1,b,b^*) +\alpha\E^{(1,b,b^*)}[ B^*(\tau_n)].
\end{align*}

Hence, letting $n \to \infty$, for each $(1,b,b^*)$
\[V^{\widehat{C}}(1,b,b^*) - U^{C,t_1}(1,b,b^*) \geq \alpha\E^{(1,b,b^*)}[ B^*(\tau)].\]
Using the scaling relation again, we get
\begin{align}\label{eqn:expectationbound}
V^{\widehat{C}}\left(\frac{1}{(b^*)^2},\frac{b}{b^*},1\right) - U^{C,t_1}\left(\frac{1}{(b^*)^2},\frac{b}{b^*},1\right) \geq \frac{\alpha}{b^*}\E^{(1,b,b^*)}[ B^*(\tau)].
\end{align}

Observe that we can also find a bound $\beta$ such that 
\begin{align*}
\beta \geq V^{\widehat{C}}(t) - U^{C,t_1}(t), \quad \text{for all} \quad t \geq 0
\end{align*}
as the difference equals $C - \widehat{C}$ for $t \geq \widehat{t_0}$ and is bounded for the compact interval
$[0,\widehat{t_0}]$. This directly yields $\beta$ as a bound on
$V^{\widehat{C}}(t,b,1) - U^{C,t_1}(t,b,1)$ which shows that
the left hand side of \eqref{eqn:expectationbound} remains uniformly bounded.
This yields \eqref{form:linbound} with $K = \beta/\alpha$ and finishes the proof.
\end{proof}

\begin{remark}
As regards the limiting case when we define $\rho$ in \eqref{01} by replacing $t_1$ by the critical value $\widehat{t}_0,$ we conjecture that we obtain $\mathbb{E} [\rho^{\frac{1}{2}}]=\infty.$ But we were not able to prove this result.
\end{remark}

\section{A pointwise version of one of Davis' inequalities}

The value function $V$ allows to derive a {\it pointwise} version of the Burkholder-Davis-Gundy inequality \eqref{V1}, which holds true in an almost sure sense rather than in expectation as stated in \eqref{V1}. This line of argument, inspired by the idea of robust superhedging from mathematical finance, is well-known (see e.g. \cite{BeSi15} and \cite{BeNu14}).

\begin{theorem}
Denote by $V=V^{\widehat{C}}$ the value function \eqref{V4} associated to the optimal constant $\widehat{C}$ and consider
the Brownian motion $B=(B(t))_{t\geq 0}$ with its (right continuous, saturated)  natural filtration $(\F(t))_{t \geq 0}$.

There is a predictable process $H(t)$ satisfying $\E[\int_0^T H^2(t) dt] < \infty$, for each $T > 0$, given a.s.\ by
\begin{equation} \label{form:hedge}
H(t) = V_b(t,B(t),B^*(t))
\end{equation}
for Lebesgue almost all $t>0$, such that, for every bounded stopping time $\tau$,
\begin{equation}\label{form:hedging}
\tau^{\frac{1}{2}}-CB^*(\tau) \leq \int_0^\tau H(t) dB(t), \qquad a.s..
\end{equation}
\end{theorem}

Before giving the proof we observe the well-known fact that \eqref{form:hedging} trivially implies \eqref{V4} by taking expectations on both sides of \eqref{form:hedging}.

\begin{proof}
Lemma \ref{lem:valuemartingale} states that the continuous process 
\[X(t) = V(t,B(t),B^*(t))\]
is a super-martingale,
starting at $X(0) = V(0,0,0)=0$. sBy Doob-Meyer we may decompose $X$ as
\begin{equation}\label{form:decomp}
X=M-A
\end{equation}
where $M$ is a continuous local martingale and $A$ is a continuous non-decreasing predictable process, and $M(0)=A(0)=0$.

In fact, $M$ is a square integrable martingale as we will show in Lemma \ref{lem:squint} in the appendix.

By martingale representation we may find a predictable process $H$ with $\E[\int_0^TH(t)^2dt] <\infty$, 
for each $T > 0$ such that
\[M(t) = \int_0^t H(u) dB(u). \]

By applying Ito to both sides of \eqref{form:decomp} we obtain the relation \eqref{form:hedge} which must
hold true, for $\P$-almost each $\omega$ and for Lebesgue almost all $t$ (the null set depending on $\omega$).
The formal application of Ito's formula can be justified using the result in Lemma \ref{lem:almostmartingale}.
\end{proof}

\section{BDG-Inequalities for general \texorpdfstring{$0 < p < 2$}{p between 0 and 2}}

The above procedure can be easily modified to obtain similar results for the inequalities
$\E[\tau^{\frac{p}{2}}] \leq C_p\E[(B^*(\tau))^p]$ with $0 < p < 2$. Lemma \ref{lem:hollow} and \ref{lem:stopreg} stay essentially the
same, with a different scaling given by
\[V(a^2t,ab,ab^*) = a^pV(t,b,b^*).\]
This leads to the PDE
\[2tV_t+bV_b+b^*V_{b^*}=pV\]
and the OIDE
\[2tV_t(t) = pV(t) + \int_0^{\infty} [V(t+s) - V(t)]g(s)ds\]
for $0 \leq t \leq t_0$ and the starting condition $V(t) = t^{\frac{p}{2}} - C$ for $t \geq t_0$.

In principle a similar analysis as in the present paper should provide explicit numerical values $\widehat{C}(p)$ and
$\widehat{t}_0(p)$, in dependence of $0 < p < 2$. We leave this task to future research.

On the other hand, for $p>2$ the present method does not seem to apply and some new idea is needed.

\section{Relation to the Burkholder constant \texorpdfstring{$\widehat{C}=\sqrt{3}$}{square root of 3}}

In this section we consider martingales also allowing for jumps and we focus (w.l.o.g.) on martingales
$(M_n)_{n=0}^N$ defined on a finite probability space $\Omega$ (see Lemma \ref{8.2} below). The BDG
inequality \eqref{V1} reads in this context as
\begin{equation}\label{form:bdgdiscrete}
\E[[M,M]_N^{\frac{1}{2}}] \leq C\E[M_N^*],
\end{equation}
where $[M,M]_n = \sum_{j=1}^{n}(M_j-M_{j-1})^2$ denotes the quadratic variation process.
It was shown by D. Burkholder \cite{Bu02} that in this context the sharp constant $\widehat{C}$ equals
$\widehat{C}=\sqrt{3}$.

One may ask for a deeper reason why we obtain a different sharp constant in \eqref{V1} for continuous
martingales as for martingales also having jumps. One reason is that the value function $V$ fails
to have a certain concavity property. 

Fix a point $d=(t,b,b^*) \in D$ as well as $\alpha>0,\beta >0$. Define the points
$d_{\alpha},d_{\beta} \in D$ by

\begin{align*}
d_{\alpha} &= (t+\alpha^2,b+\alpha,\max(b^*,|b+\alpha|)),\\
d_{\beta} &= (t+\beta^2,b-\beta,\max(b^*,|b-\beta|)).
\end{align*}

We also define $p=\frac{\beta}{\alpha+\beta}$ and $q=\frac{\alpha}{\alpha+\beta}$.

\begin{proposition} \label{prop:badpair}
There exist $x \in D$ as well as $\alpha>0$, $\beta>0$ such that
\begin{equation}\label{form:badpair}
V(d) < pV(d_{\alpha})+qV(d_{\beta}).
\end{equation}
\end{proposition}

Of course, we could verify the above proposition in a trivial way by numerically analyzing the function $V(t,b,b^*)$
and detecting explicitly some $d, \alpha$ and $\beta$. It is also clear where we should search for such a
``bad'' triple $(x,\alpha,\beta)$, namely in a neighborhood of the ``kink'' related to the ``non-smooth pasting'' (Figure \ref{fig:sub} and \ref{fig:super})
which displays a strong form of non-concavity.

But this is not our point. The purpose of the above statement is to show how the non-concavity \eqref{form:badpair} of the
value function $V$ is related to the difference between the case of continuous martingales and the case of martingales with jumps.

Also note that the equations $V_t+\frac{1}{2}V_{bb}=0$ in the interior of $D$ and $V_{b^*}=0$ on the non-stopping boundary
of $D$  (i.e. \eqref{V9a} and \eqref{p5} above), imply that in the (properly interpreted) case of \emph{infinitesimal} increments $\alpha$ and $\beta$ we do have a
``$\leq$'' in \eqref{form:badpair} above. This is the message of Lemma \ref{lem:valuemartingale}.

\begin{proof}[Proof of Proposition \ref{prop:badpair}]
Admitting the subsequent lemma, we consider a dyadic martingale $(M_n)_{n=0}^N$ starting at $M_0=0$.

Let us fix some notation: The underlying probability space is given by 
\[\Omega = \{(\omega_1,\dots,\omega_N) : \omega_n \in \{-1,1\}\}\] and the filtration $(\F_n)_{n=0}^N$ is given by $\F_n = \sigma(\omega_1,\dots,\omega_n)$. Consider the process
\[(X_n)_{n=1}^N = (V([M,M]_n,M_n,M_n^*))_{n=0}^N\] 
where $V=V^{\widehat{C}}$ is the value function \eqref{V4} associated
to the optimal constant $\widehat{C} \approx 1,27267\dots$ for continuous processes.

It may happen that $(X_n)_{n=0}^N$ is a super-martingale. In this case
\begin{align*}
0 = X_0 & \geq \E[V([M,M]_N,M_N,M_N^*)] \\
& \geq \E[[M,M]_N^{\frac{1}{2}} - \widehat{C}M_N^*],
\end{align*}
so that we obtain the inequality
\[\E[[M,M]_N^{\frac{1}{2}}] \leq \widehat{C}\E[M_N^*]. \]

However, we know that $\widehat{C} \approx 1,27267 \dots$ is smaller than the sharp constant $\widehat{C} = \sqrt{3}$ for
martingales with jumps so that there must exist some dyadic martingale $(M_n)_{n=1}^N$ such that the corresponding process
$(X_n)_{n=1}^N$ fails to be a supermartingale.

This means that there is some $0 \leq n \leq N-1$ and $\omega^{(n)} = (\omega_1,\dots,\omega_n)$ such that -- with slight abuse
of notation -- we find
\[d := ([M,M]_n(\omega^{(n)}),M_n(\omega^{(n)}),M_n^*(\omega^{(n)})) \]
as well as
\begin{align*}
d_{\alpha} &= ([M,M]_{n+1}(\omega^{(n)},1),M_{n+1}(\omega^{(n)},1),M_{n+1}^*(\omega^{(n)},1)),\\
d_{\beta} &= ([M,M]_{n+1}(\omega^{(n)},-1),M_{n+1}(\omega^{(n)},-1),M_{n+1}^*(\omega^{(n)},-1)).
\end{align*}
such that inequality \eqref{form:badpair} holds true.
\end{proof}

For the following Lemma recall that a martingale is \emph{dyadic} if the increment 
$M_{n+1}-M_n$ can attain at most two values, conditionally
on $\sigma(M_1,\dots,M_n)$,.

\begin{lemma}\label{8.2}
For a constant $C>0$ the following are equivalent:
\begin{enumerate}[(i)]
\item Every dyadic martingale $(M_n)_{n=0}^N$ satisfies \eqref{form:bdgdiscrete}.
\item Every martingale $(M_n)_{n=0}^N$ defined on
a finite probability space satisfies \eqref{form:bdgdiscrete}.
\item Every $L^2$ bounded martingale $(M_t)_{0 \leq t \leq T}$ satisfies \eqref{form:bdgdiscrete}.
\end{enumerate}
\end{lemma}

\begin{proof}
The equivalence (ii) $\Leftrightarrow$ (iii) is standard but for the convenience of the
reader we will recall the argument for the non-trivial implication (ii) $\Rightarrow$ (iii).

First, we can reduce the problem to discrete $L^2$-martingales: Fix an
$L^2$-bounded martingale $M = (M_t)_{0 \leq t \leq 1}$, based on a filtered
probability space $(\Omega,\F,(\F_t)_{0 \leq t \leq T},\P)$ and consider the
martingales $(M_{k2^{-n}})_{k=0}^{2^n}$ for $n \in \N$. If they fulfill \eqref{form:bdgdiscrete},
then letting $n \to \infty$ yields that $M$ satisfies \eqref{form:bdgdiscrete}.

Now, fix an $L^2$-bounded martingale $M = (M_n)_{0 \leq n \leq N}$ on a filtered
probability space $(\Omega,\F,(\F_n)_{0 \leq n \leq N},\P)$. Consider the net of
finite subfiltrations $(\F_n)_{0 \leq n \leq N}$ of this filtration and their associated
martingales $M^\G = (M^\G_n)_{n=0}^{N}$ (i.e. $M^\G_n = \E[M_n | \G_n]$). By (ii), every $M^\G$
satisfies \eqref{form:bdgdiscrete}. The $L^2$ limit of $M^\G$ is $M$ and (iii) follows.

The implication (ii) $\Rightarrow$ (i) is trivial.

To show (i) $\Rightarrow$ (ii) first observe that without loss of generality
we can assume $M_0$ to be deterministic. First, we can translate the martingale
such that it has mean $0$. Second, if $M_0$ is random, then define a martingale
$(M'_n)_{n=0}^{N+1}$ with $M'_0 := 0$ and $M'_i := M_{i-1}$ for $i > 0$. Then
$[M',M']_{N+1}=[M,M]_N$ and $M^{'*}_{N+1}=M^*_N$.

Now, suppose first that $(M_n)_{n=0}^1$ is just a one step martingale on a
finite probability space $\Omega$ with $M_0 = 0$.
We then have that $M_1$ is a finitely valued random variable with $\E[M_1]=M_0=0$.

By possibly passing to a bigger (still finite) $\Omega$ we may find a partition $(A_1, \dots, A_p)$
of $\Omega$ such that $M_1$ takes at most $2$ values on each $A_j$ and
\[\E[M_1\mathbbm{1}_{A_j}] = 0, \quad j=1,\dots,p. \]
We now define a dyadic martingale $(\tilde{M}_j)_{j=0}^p$ by $\tilde{M}_0 = 0$ and
\[\tilde{M}_j - \tilde{M}_{j-1} := M_1\mathbbm{1}_{A_j}, \quad j=1,\dots,p. \]
Clearly the variables $([M,M]_1,M^*_1)$ and $([\tilde{M},\tilde{M}]_p,\tilde{M}^*_p)$ are equal in law.
This shows that, at least for $N=1$, we may associate to every finitely values martingale $(M_n)_{n=0}^1$
a dyadic martingale $(\tilde{M}_j)_{j=0}^p$ such that \eqref{form:bdgdiscrete} holds true for $M$ if and only if it does
so for $\tilde{M}$.

It is rather obvious how to continue the above construction in an inductive way so that we may associate to each finitely
valued martingale $(M_n)_{n=0}^N$ a dyadic martingale $(\tilde{M}_j)_{j=0}^p$ such that
$([M,M]_N,M^*_N)$ and $([\tilde{M},\tilde{M}]_p,\tilde{M}^*_p)$ are equal in law.
This readily shows (i) $\Rightarrow$ (ii).
\end{proof}

\appendix

\section{The Martingale Property of the Value Process and Square Integrability}

\begin{proof}[Proof of Lemma \ref{lem:valuemartingale}]

We denote the function appearing on the right side of \eqref{p2} by $v(t,b^*)$:
\begin{align*}
v(t,b^*):= t^{\frac{1}{2}}- C b^*, \qquad t,b^* \geq 0.
\end{align*}
For fixed $(t,b,b^*) \in D$ and $t \leq T$ we denote by $V^T(t,b,b^*)$ the value function defined similarly as in \eqref{V4}, but where we only allow for stopping times $\tau \in \mathcal{T}(t)$ which are bounded by $T$. Clearly $V^T(t,b,b^*)$ increases to $V(t,b,b^*),$ as $T \to \infty$, pointwise for $(t,b,b^*) \in D.$ Also note that $V(T,b,b^*)=v(T,b^*).$ Fix $(t,b,b^*)$ and a bounded stopping time $t \leq \sigma \leq \tau$. We have to show that
\begin{align}
V(t,b,b^*)= \mathbb{E}^{(t,b,b^*)}[ V(\sigma, B(\sigma), B^*(\sigma) ) ]
\end{align}

By the monotone convergence theorem it will suffice to show that
\begin{align}\label{P1}
\lim_{T\to\infty} V^T(t,b,b^*)= \lim_{T\to \infty} \mathbb{E}^{(t,b,b^*)}[V^T(\sigma^T, B(\sigma^T), B^*(\sigma^T)]
\end{align}
for $\sigma^T := \sigma \wedge \tau^T$ where $\tau^T$ is the stopping time defined conditionally on $(t,b,b^*)$ by
\begin{align*}
\tau^T=\inf \{u \geq t:V^T(u,B(u), B^*(u))=v(u,B^*(u))\}.
\end{align*}
We then have that $\tau^T$ is bounded by $T$ and increases a.s.~to $\tau$. The crucial property is
\begin{align*}
V^T(t,b,b^*)= \mathbb{E}^{(t,b,b^*)}[v(\tau^T, B^*(\tau^T))],
\end{align*}
and, more generally, for any stopping time $t\leq \rho \leq \tau^T,$
\begin{align*}
V^T(\rho, B(\rho), B^*(\rho) ) = \mathbb{E}[v(\tau^T, B^*(\tau^T))| \mathcal{F} (\rho)].
\end{align*}
This classical result can be found in \cite[Theorem 2.2]{PeSh06}.
Putting this together and taking $\rho = \sigma^T$, we obtain \eqref{P1}.

The proof of the supermartingale property which still holds true, after time $\tau$ is identical with an inequality instead of an equality.

\end{proof}

We can even show that the value process is bounded in $L^2$ up to some fixed time $T$:

\begin{lemma}\label{lem:squint}
The supermartingale $(X(t))_{0 \leq t \leq T}$ given by
\[X(t) := V(t,B(t),B^*(t)) \]
is uniformly bounded from above and bounded in $L^2$. Furthermore the martingale component of
its Doob-Meyer decomposition $X(t) = M(t) - A(t)$ is also bounded in $L^2$ and we obtain the
following quantitative estimates for every stopping times $\sigma$ with $0 \leq \sigma \leq T$:
\begin{enumerate}[(i)]
\item $X(\sigma) \leq T^{\frac{1}{2}},$
\item $\E[X(\sigma)^2] \leq KT$,
\item $\E[M(\sigma)^2] \leq \E[M(T)^2] < \infty$.
\end{enumerate}
\end{lemma}

\begin{proof}
We first observe that
$V(t,b,|b|) - (t^{\frac{1}{2}} - C|b|) \leq V(0,b,|b|) -(-C|b|) \leq C|b|$ holds as a consequence
of the proof of Lemma \ref{lem:stopreg} which gives us the estimate $V(t,b,|b|) \leq t^{\frac{1}{2}}$. Letting $t\to 0$
this also imples $V(t,0,0) \leq t^{\frac{1}{2}}$ and we can show (i) by
\[V(t,b,b^*)\leq V(t,0,b^*) \leq V(t,0,0) \leq t^{\frac{1}{2}}.\]
$V$ is monotone increasing in $t$ and monotone decreasing in $|b|$ and $b^*$. So we can observe
for the positive part of $X$ that
\[\E[(X(\sigma)_+)^2] \leq (V(T,0,0)_+)^2 \leq T. \]
For the negative part $X(t)_-$ we can use $V(t,b,b^*) \geq t^{\frac{1}{2}} -Cb^* \geq -Cb^*$ and estimate
\[\E[(X(\sigma)_-)^2] \leq C^2\E[B^*(T)^2]. \]
In summary we have
\[\E[X(\sigma)^2] \leq T+C^2\E[B^*(T)^2]. \]

To show the last assertion, we now split $X$ into a sum of bounded processes in the following way. Define
the stopping times $(\sigma_n)_{n=0}^{\infty}$ by
\[\sigma_n = \inf\{t : |B(t)| = 2^n\} \wedge T, \]
and define the processes $X_n$, obtained by starting $X$ at time $\sigma_{n-1}$ and stopping it
at time $\sigma_n$:
\[X_n(t) = \presuper{\sigma_{n-1}}{X}^{\sigma_n}(t) = (X_{t\wedge \sigma_n} - X_{\sigma_{n-1}})\mathbbm{1}_{\llbracket \sigma_{n-1},T\rrbracket}(t). \]
Of course, we have $X = \sum_{n=1}^{\infty} X_n$ and the trajectories of $(X_n(t))_{0 \leq t \leq T}$ are only different from zero on the set $(\{\sigma_{n-1} < T\})_{n=1}^{\infty}$. 

The probability of these events can be estimated by
\begin{align}\label{app1}
\P[\sigma_{n-1} < T] \leq c_1 e^{-c_2 2^{2n}},
\end{align}
for some constants $c_1 = c_1(T)$ and $c_2 = c_2(T)$.

Using a classical inequality on uniformly bounded supermartingales (apparently due to P. Meyer \cite{Me72}) we obtain that each $M_n$
is a square integrable martingale whose norm can be estimated by
\begin{align}\label{app2}
||M_n(T)||^2_{L^2(\P)} \leq c_32^{c_4 n}\P[\sigma_{n-1} < T]
\end{align}
for some constants $c_3, c_4$ depending only on $T$.

Combining \eqref{app1} and \eqref{app2}, we deduce that
\[||M(T)||^2_{L^2(\P)} = \sum_{n=1}^{\infty}||M_n(T)||^2_{L^2(\P)} < \infty.\]
\end{proof}

For the convenience of the reader we spell out the message of Meyer's Theorem \cite[Theorem 46]{Me72} in
the present context.

\begin{theorem}[Meyer] \label{A.2}
Let $X = (X(t))_{0 \leq t \leq T}$ be a uniformly bounded supermartingale
\[ ||X||_{\infty} := \sup_{0 \leq t \leq T} ||X(t)||_{L^{\infty}} \leq c < \infty.\]
Denoting by $X = M-A$ its Doob-Meyer decomposition we get that $M$ is a square integrable
martingale whose norm can be estimated by $||M||_2^2 := ||M(T)||_{L^2}^2 \leq 18c^2$
\end{theorem}

\begin{proof}
By standard approximation results it will suffice to show the result for a super-martingale $X = (X(n))_{n=0}^N$
in finite discrete time. Note that in this case we have $A(0) = 0$ and
\[A(n+1)-A(n) = \E[X(n+1)-X(n)|\F(n)], \quad n=1,\dots,N \]
so that
\[||\Delta A(n)||_{L^{\infty}} \leq 2c.\]
We may telescope $A(N) = \sum_{n=1}^N(A(n)-A(n-1))$ to obtain
\begin{align*}
A(N)^2 &= 2 \sum_{n=1}^N(A(n)-A(n-1))\sum_{j=n}^N(A(j)-A(j-1)) \\
& = 2 \sum_{n=1}^N (A(n)-A(n-1))(A(N)-A(n-1)).
\end{align*}
By taking expectations we get
\begin{align*}
\E[A(N)^2] &= 2 \sum_{n=1}^N \E[\E[(A(n)-A(n-1))(A(N)-A(n-1))|\F(n-1)]] \\
&=2 \sum_{n=1}^N\E[ (A(n)-A(n-1))\E[A(N)-A(n-1)|\F(n-1)]].
\end{align*}
The final term is uniformly bounded as 
\[\E[A(N)-A(n-1)|\F(n-1)] = \E[X(N)-X(n-1)|\F(n-1)] \leq 2c. \]
This yields
\begin{align*}
\E[A(N)^2] &\leq 4c \sum_{n=1}^N \E[A(n)-A(n-1)] \\
&=4c \E[X(N)-X(0)] \leq 8c^2.
\end{align*}
To obtain a bound for $||M||_2$ we use the relation $M=X+A$
and $||X(N)||_{L^\infty} \leq c$ to get
\begin{align*}
\E[M(N)^2] \leq 2\E[A(N)^2] + 2c^2 \leq 18c^2.
\end{align*}
\end{proof}

\section{Some facts on the stopping time of first leaving a corridor}

We discuss the first exit time of the interval $[-h,2+h]$ for some $h>0$ for a standard Brownian motion $B$ started at $B(0)=0$.
\[\sigma^h := \inf \{t : |B(t)-1|=1+h\}.\]
This stopping time has a well-known density and a well-known Laplace-Transform $\L$ (see e.g.\ \cite[Section 2.2.8.C]{KaSh88} given by 
\[\L(\theta) = \frac{\cosh(\sqrt{2\theta})}{\cosh((1+h)\sqrt{2\theta})}.\]
One can calculate the expected value of the stopping time by noting that $\E[\sigma^h] = \E[B_{\sigma^h}^2]$
so that $\E[\sigma^h] =h(2+h)$. The Jensen-inequality directly implies that the fractional moments of order less than 
$1$ also exist and one has the estimate $\E[(\sigma^h)^{\frac{1}{2}}] \leq (h(2+h))^{\frac{1}{2}}$.
This motivates to conjecture that $\frac{\E[(\sigma^h)^{\frac{1}{2}}]}{h} \to \infty$ for $h \to 0$. We can get an expression for this moment using
the Laplace-transform above by the formula 
$\E[(\sigma^h)^{\frac{1}{2}}] = \frac{-1}{\Gamma(\frac{1}{2})}\int_{0}^{\infty} \theta^{-\frac{1}{2}}\L'(\theta) \,d\theta$, which
we can evaluate to 
\[\E[(\sigma^h)^{\frac{1}{2}}] = \sqrt{\mfrac{2}{\pi}}\left[\int_0^{\infty} \mfrac{\sinh(hu)}{u\cosh((1+h)u)^2} \,du + 
h\int_0^{\infty} \mfrac{\cosh(u)\tanh((1+h)u)}{u\cosh((1+h)u)} \,du\right] \]
where a substitution $u=\sqrt{2\theta}$ was used. It can be shown that both of these integrals are in fact finite
for positive $h$, but we are only interested in the above limiting behavior of this expression. To see this first note that
the integrands of both integrals are always positive. Furthermore the hyperbolic tangent converges to $1$. Therefore we can fix a constant $K$ such that for $h < 1$ we have that $\tanh((1+h)u) \geq \frac{1}{2}$ for $u \geq K$.
Putting this together we make the following estimate:
\begin{align*}
\frac{\E[(\sigma^h)^{\frac{1}{2}}]}{h} &\geq \frac{1}{\sqrt{2\pi}} \int_K^{\infty} \frac{\cosh(u)}{u\cosh((1+h)u)}\,du \\
& \geq \frac{1}{2\sqrt{2\pi}}\int_K^{\infty} \frac{e^{-hu}}{u} \, du = \frac{1}{2\sqrt{2\pi}}\int_{hK}^{\infty} \frac{e^{-u}}{u} \,du.
\end{align*}
The last expression now obviously diverges for $h \to 0$. The following Lemma which we will need later on uses the above observations:

\begin{lemma} \label{lem:superlinear}
Let $K>0$ be an arbitrary constant. There exist $t,h>0$ such that
\[\E[(t+\sigma^h)^{\frac{1}{2}}-t^{\frac{1}{2}}] \geq Kh.\]
\end{lemma}

\begin{proof}
By the above observations we can choose $h$ small enough to obtain $\E[(\sigma^h)^{\frac{1}{2}}] \geq Kh+1$.
We also have $\E[(\sigma^h)^{\frac{1}{2}}]-\E[(t+\sigma^h)^{\frac{1}{2}}-t^{\frac{1}{2}}] < 1$ for small enough
$t$ by monotone convergence of $(t+s)^{\frac{1}{2}}-t^{\frac{1}{2}}$ to $s^{\frac{1}{2}}$ for $t \to 0$.
\end{proof}

We can now proceed to show the following facts about $t_0$ to prove the final statement of Lemma \ref{lem:stopreg}.

\begin{lemma} \label{lem:Vdec}
Let $C \geq \widehat{C}$ and $V(t,b,b^*)$ the corresponding value function. The map $t \mapsto V(t,1,1)-(t^{\frac{1}{2}} - C)$ is
decreasing and if $V(1,1,1) > 1-C$ it follows that $V(1,1,1) - (1 - C) < V(0,1,1) - (-C)$.
\end{lemma}

\begin{proof}
We need a quantitative version of Lemma \ref{lem:valueprops} (ii) which already shows that $t \mapsto V(t,1,1) - (t^{\frac{1}{2}} - C)$ is decreasing. 
First choose some $\nu \geq 1$ to be a bounded 
stopping time which achieves 
\begin{equation}
V(1,1,1)-\E^{(1,1,1)}[\nu^{\frac{1}{2}} - C B^*(\nu)] < \varepsilon \label{form:Vapprox}
\end{equation}
such that
\begin{equation}
\P[\nu > 1+2\varepsilon] > (2\varepsilon)^{\frac{1}{2}} \label{form:nuproper}
\end{equation}
for arbitrary $\varepsilon > 0$. It is clear by definition of $V$ that there exists a stopping time which satisfies \eqref{form:Vapprox}. Suppose
there is no appropriate stopping time such that \eqref{form:nuproper} is satisfied, then there is an optimizing sequence of
bounded stopping times which converge to $1$ in probability and thus a subsequence which converges almost surely.
This would imply that $V(1,1,1) = 1-C$ which contradicts the assumptions of the Lemma.

Now, we can consider the stopping time $\nu$ as a randomized stopping time
with respect to the filtration $(\G^{(1)}(u))_{u \geq 1}$. The shifted stopping time $\nu':=\nu-1$ is then a randomized stopping time
with respect to $(\G^{(0)}(u))_{u \geq 0}$. We can now estimate

\begin{align*}
[V(0,1,1)&-(-C)] - [V(1,1,1)-(1-C)] \\
& \geq \E^{(0,1,1)}[\nu'^{\frac{1}{2}} - C B^*(\nu')] - \E^{(1,1,1)}[\nu^{\frac{1}{2}} - C B^*(\nu)] + 1 - \varepsilon \\
& = \E[(\nu')^{\frac{1}{2}} - (1+\nu')^{\frac{1}{2}} + 1] - \varepsilon \\
& \geq \E[((\nu')^{\frac{1}{2}} - (1+\nu')^{\frac{1}{2}} + 1)\mathbbm{1}_{\nu' > 2\varepsilon}] - \varepsilon \\
& \geq [(2\varepsilon)^{\frac{1}{2}} - ((1+2\varepsilon)^{\frac{1}{2}} - 1)](2\varepsilon)^{\frac{1}{2}}  - \varepsilon \\
& > \varepsilon - \sqrt{2}\varepsilon^{\frac{3}{2}} > 0.
\end{align*}
To get from the second to the third line, we used that by definition $\nu = \nu'+1$ and $\E^{(0,1,1)}[B^*(\nu')] = \E^{(1,1,1)}[B^*(\nu)]$.
We dropped the superscript to emphasize that we now view $\nu'$ as a stopping time with respect to
the filtration $(\G^{(0)}(u))_{u \geq 0}$. To obtain the fourth and the fifth line in the derivation, we observe that the
map $t \mapsto t^{\frac{1}{2}} - (1+t)^{\frac{1}{2}} + 1$ is non-negative and non-decreasing, and use 
\eqref{form:nuproper}. The sixth line can be derived
by noting that $s \mapsto (1+s)^{\frac{1}{2}} - 1$ is concave and thus lies completely under its tangent at $s=0$.
The last inequality holds for $\varepsilon$ small enough.
In the same way one can actually show that as long as the spread $V(t,1,1) - (t^{\frac{1}{2}} - C)$ is strictly
positive, it is also strictly decreasing.
\end{proof}

\begin{lemma} \label{lem:t0int}
Let $C \geq \widehat{C}$ and $t_0 = t_0(C)$ the critical point separating $S$ from $NS$. Then
\begin{enumerate}
\item $t_0 >0$ and
\item $t_0 < \infty$.
\end{enumerate}
\end{lemma}

\begin{proof}
(1): Assume $t_0 = 0$. This means that we actually have 
\[V(t,b^*,b^*) = t^{\frac{1}{2}} - Cb^*\]
for all $t > 0$.
We then obtain for arbitrary $t,h >0$ that
\[\E^{(t,1,1)}[(\sigma^h)^{\frac{1}{2}}-C(1+h)] \leq (t)^{\frac{1}{2}} -C\]
by the supermartingale property of the value-process. This is a contradiction to Lemma \ref{lem:superlinear} for small 
enough $t$ and $h$.

(2): Assume $t_0 = \infty$. This means that $V(t,1,1) > t^{\frac{1}{2}} - C$ everywhere. As $C \geq \widehat{C}$
we also have $V(0,1,1) \leq 0$. By Lemma \ref{lem:Vdec} we can set $\alpha:= V(1,1,1)-1+C < C$. Now
fix some $h > \frac{\alpha}{C-\alpha}$ and $t > (1+h)^2$. We can then make the following estimate, 
where we use twice the fact that the function $t \mapsto V(t,b,|b|) - (t^{\frac{1}{2}} -C|b|)$ is decreasing
and $\sigma^h$ is defined as before:
\begin{align*}
V(t,1,1) &= \E^{(t,1,1)}[V(\sigma^h,1+h,1+h)] \\
&\leq V(t,1+h,1+h) + \E^{(t,1,1)}[(\sigma^h)^{\frac{1}{2}} - t^\frac{1}{2}]\\
&= (1+h)V\left(\frac{t}{(1+h)^2},1,1\right) + \E^{(t,1,1)}[(\sigma^h)^{\frac{1}{2}} - t^\frac{1}{2}] \\
&\leq t^{\frac{1}{2}} -(1+h)(C-\alpha) + \E^{(t,1,1)}[(\sigma^h)^{\frac{1}{2}} - t^\frac{1}{2}] \\
& < V(t,1,1) - h(C-\alpha) + \alpha + \E^{(t,1,1)}[(\sigma^h)^{\frac{1}{2}} - t^\frac{1}{2}] 
\end{align*}
Now we can eliminate $V(t,1,1)$ on both sides and note that $h$, $C$ and $\alpha$ do not depend on
$t$. The last term however goes to $0$ for $t \to \infty$ by dominated convergence (since
$0 \leq \sqrt{t+\sigma^h} - \sqrt{t} \leq \sqrt{\sigma^h}$ and $\sqrt{t+\sigma^h} - \sqrt{t} \searrow 0$). This leads to
the desired contradiction.
\end{proof}

\section*{Acknowledgements}

Thanks go to an extremely helpful referee who provided a very careful report. The paper in its present form owes much to his suggestions. We also thank Mathias Beiglböck for his insight and advice in the course of many discussions on the present paper and
Josef Teichmann for pointing out to us Meyer's inequality (Theorem \ref{A.2}).
 
\bibliography{biblio}{}

\begin{thebibliography}{10}

\bibitem{AlKy05}
L.~Alili and A.~E. Kyprianou.
\newblock Some remarks on first passage of {L}{\'e}vy processes, the {A}merican
  put and pasting principles.
\newblock {\em Ann. Appl. Probab.}, 15(3):2062--2080, 2005.

\bibitem{AsAvPi04}
S.~Asmussen, F.~Avram, and M.~R. Pistorius.
\newblock Russian and american put options under exponential phase-type
  {L}{\'e}vy models.
\newblock {\em Stochastic Processes and their Applications}, 109(1):79--111,
  2004.

\bibitem{BaCh77}
J.~R. Baxter and R.~V. Chacon.
\newblock Compactness of stopping times.
\newblock {\em Z. Wahrscheinlichkeitstheorie und Verw. Gebiete},
  40(3):169--181, 1977.

\bibitem{BeNu14}
M.~Beiglb{\"o}ck and M.~Nutz.
\newblock Martingale inequalities and deterministic counterparts.
\newblock {\em Electron. J. Probab.}, 2014.

\bibitem{BeSi15}
M.~Beiglb{{\"o}}ck and P.~Siorpaes.
\newblock Pathwise versions of the {B}urkholder-{D}avis-{G}undy inequality.
\newblock {\em Bernoulli}, 21(1):360--373, 2015.

\bibitem{BoLe02}
S.~I. Boyarchenko and S.~Z. Levendorskii.
\newblock Perpetual american options under l{\'e}vy processes.
\newblock {\em SIAM Journal on Control and Optimization}, 40(6):1663--1696,
  2002.

\bibitem{Bu02}
D.~Burkholder.
\newblock The best constant in the {D}avis inequality for the expectation of
  the martingale square function.
\newblock {\em Transactions of the American Mathematical Society},
  354(1):91--105, 2002.

\bibitem{BuGu70}
D.~L. Burkholder and R.~F. Gundy.
\newblock Extrapolation and interpolation of quasi-linear operators on
  martingales.
\newblock {\em Acta Math.}, 124:249--304, 1970.

\bibitem{DaHo04}
R.~C. Dalang and M.-O. Hongler.
\newblock The right time to sell a stock whose price is driven by markovian
  noise.
\newblock {\em Annals of Applied Probability}, pages 2176--2201, 2004.

\bibitem{Da70}
B.~Davis.
\newblock On the intergrability of the martingale square function.
\newblock {\em Israel Journal of Mathematics}, 8(2):187--190, 1970.

\bibitem{KaSh88}
I.~Karatzas and S.~Shreve.
\newblock {\em Brownian motion and stochastic calculus}, volume 113 of {\em
  Graduate Texts in Mathematics}.
\newblock Springer-Verlag, New York, 1988.

\bibitem{Me72}
P.-A. Meyer.
\newblock {\em Martingales and Stochastic Integrals I}.
\newblock Springer-Verlag, Berlin, 1972.

\bibitem{PeSh06}
G.~Peskir and A.~Shiryaev.
\newblock {\em Optimal stopping and free-boundary problems}.
\newblock Springer, 2006.

\bibitem{PeSh00}
G.~Peskir and A.~N. Shiryaev.
\newblock Sequential testing problems for {P}oisson processes.
\newblock {\em Annals of Statistics}, pages 837--859, 2000.

\bibitem{Sh67}
L.~Shepp.
\newblock A first passage problem for the {W}iener process.
\newblock {\em The Annals of Mathematical Statistics}, pages 1912--1914, 1967.

\end{thebibliography}
\bibliographystyle{abbrv}

\end{document}